\RequirePackage{fix-cm}
\documentclass[smallextended,envcountsect]{svjour3}       % onecolumn (second format)
\smartqed  % flush right qed marks, e.g. at end of proof
\usepackage{graphicx}
\usepackage{mathptmx}      % use Times fonts if available on your TeX system

\usepackage{amsfonts}
\usepackage{amssymb,amsmath,enumerate}
\usepackage{latexsym}
\usepackage{fullpage}
\usepackage{hyperref}

\newtheorem{defn}[theorem]{Definition}

\newcounter{tenumerate}

\def\P{\mathbb{P}}

\renewcommand{\epsilon}{\varepsilon}

\DeclareMathOperator{\var}{Var} 

\DeclareMathOperator{\arccoth}{arccoth}

\newcommand{\E}{{\mathbb E}}
\newcommand{\remove}[1]{}

\renewcommand{\leq}{\leqslant}
\renewcommand{\geq}{\geqslant}

\def\XXint#1#2#3{{\setbox0=\hbox{$#1{#2#3}{\int}$}
		\vcenter{\hbox{$#2#3$}}\kern-.5\wd0}}

\usepackage{bm}
\usepackage{graphicx}
\usepackage{epstopdf}

\journalname{Journal of Theoretical Probability}

\begin{document}

\title{Phase transition for accessibility percolation on hypercubes}
\thanks{Partially supported by NSF grant DMS-1313596.}
\author{Li Li}

\institute{Li Li \at
              University of Chicago \\
              Tel.: +1-607-3794673\\
              \email{lili@galton.uchicago.edu}
}

\date{Received: date / Accepted: date}

\maketitle

\begin{abstract}
In this paper, we consider accessibility percolation on hypercubes, i.e., we place i.i.d.\ uniform $[0,1]$ random variables on vertices of a hypercube, and study whether there is a path connecting two vertices such that the values of these random variables increase along the path. We establish a sharp phase transition depending on the difference of the values at the two endpoints, and determine the critical window of the phase transition. Our result completely resolves a conjecture of Berestycki, Brunet and Shi (2014).

\end{abstract}

\section{Introduction}

For $N\in \mathbb N$, let $H_N= \{0,1\}^N$ be a hypercube where two vertices are connected by an \emph{undirected} edge if their Hamming distance, i.e. the number of coordinates at which they differ, is precisely 1. Let $\{X_v: v\in H_N\}$ be i.i.d.\ random variables uniformly distributed in $[0,1]$. We say that a path in $H_N$ is accessible if the associated random variables $X_v$'s are increasing along the path. For $u, w\in H_N$, we say that $w$ is accessible from $u$ if there exists at least one accessible path from $u$ to $w$. In this paper, we show that the conditional accessible probability (from $u$ to $w$) given that $X_{u}=a$ and $X_{w}=b$ ($0\leq a<b\leq 1$) admits a sharp phase transition, in a sense made precise in Theorem \ref{thm-main} below. By symmetry, the conditional accessible probability with fixed $a$ and $b$ depends only on the Hamming distance between $u$ and $w$. Therefore, we fix $0<\beta \leq 1$ and without loss of generality consider the case when $u=(0,0,\cdots,0)$ and $w=(1,1,\cdots,1,0,0,\cdots,0)$ (here the number of 1's in $w$ is $[\beta N]$). Furthermore, since subtracting $a$ from all $X_v$'s does not change the accessibility from $u$ to $w$, we can also assume without loss of generality that $a = 0$ and $b = x$ (where $x$ may depend on $N$). Our main result is summarized in the following theorem.

\begin{theorem} \label{thm-main}
	Let $f(x)=(\sinh x)^\beta(\cosh x)^{1-\beta}$, and let $x_0$ be the unique number such that $f(x_0) = 1$. Define
	$x_c(N)=x_0-{\frac 1{f'(x_0)}}{\frac {\ln N}N}$. For any sequence $\epsilon_N$ such that $N \epsilon_N  \to \infty$, we have
	\begin{align}
	&\lim\limits_{N\to \infty} \P( w \mbox{ is accessible from } u \mid X_{u} = 0, X_{w} = x_c-\epsilon_N)=0\,,  \label{eq-mainthm-upper}\\
	&\lim\limits_{N\to \infty} \P( w \mbox{ is accessible from } u \mid X_{u} = 0, X_{w} = x_c+\epsilon_N)=1\,. \label{eq-mainthm-lower}
	\end{align}
	In addition, for all $\Delta>0$, there exist $0<c_1<c_2<1$ (where $c_1$ and $c_2$ depend only on $\Delta$) such that for all $N\in \mathbb N$
	\begin{equation}\label{eq-mainthm-critical}
	c_1 \leq \P( w \mbox{ is accessible from } u \mid X_{u} = 0, X_{w} = x_c + \epsilon_N) \leq c_2\,, \mbox{ if }  |N\epsilon_N| \leq \Delta\,.
	\end{equation}
\end{theorem}

\begin{remark}
	A few days before the post of this article, we noted that a paper \cite{Martinsson15} was posted in January 2015, which proved the version of \eqref{eq-mainthm-lower} (without analyzing the critical window for the phase transition) for the case of $\beta \geq 0.002$. While we acknowledge the priority of \cite{Martinsson15}, we emphasize that our work was carried out independently; our method is rather different and allows us to derive the result for all $0<\beta\leq 1$. 
\end{remark}

Accessibility percolation on hypercubes with backsteps (i.e., when the hypercube graph is undirected as we have assumed at the beginning) was studied in \cite{BBS14}, where they proved \eqref{eq-mainthm-upper} and conjectured \eqref{eq-mainthm-lower} (both in a slightly weaker form). Our Theorem~\ref{thm-main} completes the picture and describes a sharp phase transition for this problem.

An analogue of Theorem~\ref{thm-main} on accessibility percolation on hypercubes without backsteps (i.e., when the edges of the hypercube are directed toward the vertex with the greater number of ones) was established by  \cite{HM14}. Under the same setting, \cite{BBS13} gives the asymptotic distribution of the number of accessible paths when $x$ is in a different regime. Accessibility percolation has also been studied on N-ary trees \cite{NK13,RZ13,Chen14} and on spherically symmetric trees \cite{CGR14}. In addition, the Hamiltonian increasing path on the complete graph was studied in \cite{LL14}.

Our study on accessibility percolation is motivated by the NK fitness landscapes, which were introduced in \cite{KL87,Kauffman} as a class of models for biological evolution. In the NK fitness model, we consider $H_N$ corresponding to, e.g.,  nucleobases in a DNA sequence. Let $F$ be a distribution. Given $K\leq N$, let $Y_{i, \tau}$ be i.i.d.\ random variables with distribution $F$ for all $1\leq i\leq N$ and $\tau \in H_K$. For $\sigma\in H_N$,  the fitness of $\sigma$ is then defined to be
$X_\sigma = \mbox{$\sum_{i=1}^{N}$} Y_{i, (\sigma_i, \ldots, \sigma_{i+K-1})}$
(where the addition in the subscript is understood as modulo of $N$).  Since the gene favors better fitness, it is natural to consider an adaptive walk on space $H_N$ such that the corresponding fitness increases until the walk is frozen at a local maximum.  Theorem~\ref{thm-main} is a preliminary step toward understanding the adaptive walk on the NK fitness model. Indeed, our model (with i.i.d.\ fitness for each vertex in $H_N$) corresponds to the case when $K = N$ (the distribution $F$ does not play a role when considering increasing paths as long as $F$ is continuous).

\section{Accessibility percolation: antipodal case}

For clarity of presentation, in the current section we give a proof of Theorem~\ref{thm-main} in the antipodal case when $\beta=1$, i.e., when $u=\vec{0}=(0,0,\cdots,0)$ and $w=\vec{1}=(1,1,\cdots,1)$. In Section~\ref{sec:general}, we modify the arguments and give a proof of Theorem~\ref{thm-main} in the general case when $0<\beta<1$. In both sections, the probability measure $\P$ stands for the conditional probability given $X_{u}= 0$ and $X_{w} = x$, unless otherwise specified. Recall that a path from $u$ to $w$ is accessible if the $X_v$'s (including $X_{u}$ and $X_{w}$) along the path are increasing. Denote by $Z_{N,x}$ the number of such accessible paths. Throughout the paper, we sometimes write \emph{with high probability} for brevity to mean with probability tending to 1 as $N\to \infty$.

\subsection{Proof of the upper bound} \label{subsec:Proof of the upper bound}

In this subsection we give a proof of \eqref{eq-mainthm-upper} in the antipodal case (the general case is similar). Note that Lemma \ref{lem-M-k-n}  below (which implies \eqref{eq-upper-bound-first-moment} in Corollary \ref{cor-first-moment} and therefore \eqref{eq-mainthm-upper} in the general case) has already been proved in \cite{BBS14}. Here we give a different proof of Lemma \ref{lem-M-k-n}, by relating the original model to a more tractable one (i.e. $\mu_{k, n }$), and this connection will also be useful in later proofs. We start with a number of definitions.

\begin{defn}\label{def-path}
	We say that a path (not necessarily self-avoiding) in $H_N$ has length $\ell$ if it visits $(\ell-1)$ inner vertices (a vertex is counted each time it is visited, starting and ending points are excluded). For $n, \ell\in \mathbb N$, let $\mathcal  M(n,\ell)$ be the collection of paths (not necessarily self-avoiding) of length $\ell$ from $\vec{0}_N = (0,0,\cdots,0)$ to $ (\vec 1_{n}, \vec 0_{N-n}) = (1,1,\cdots,1,0,0,\cdots,0)$ (where there are $n$ 1's in $(\vec 1_{n}, \vec 0_{N-n})$). Write $M(n,\ell) = |\mathcal M(n,\ell)|$.
\end{defn}

\begin{defn}\label{def-sequence}
	For $n, \ell\in \mathbb N$, let $\mathcal S(n,\ell)$ be the collection of integer sequences $(a_1, \ldots, a_\ell) \in \{1, \ldots, N\}^\ell$ such that $|\{1\leq i\leq \ell: a_i = k\}|$ is odd for $1\leq k\leq n$ and even for $n+1\leq k\leq N$. In addition, for $1\leq k\leq N$, let $\mathcal S_k(n, \ell) \subseteq \mathcal S(n, \ell)$ contain all sequences in $\mathcal S(n, \ell)$ such that the last number $a_\ell$ is $k$ and let $\mathcal S_k(n) = \cup_{\ell \in \mathbb N} \mathcal S_{k}(n, \ell)$.
\end{defn}

For each path (not necessarily self-avoiding) $v_0, v_1, \ldots, v_\ell$ in $H_N$ of length $\ell$, we associate a sequence of integers $(a_1, \ldots, a_\ell)$ where $a_i$ is the coordinate at which $v_{i-1}$ and $v_i$ differ. We observe that the association is a bijection between $\mathcal M(n, \ell)$ and $\mathcal S(n, \ell)$.

\begin{remark}
	In the following we will sometimes call the sequence $(a_1, \ldots, a_\ell)$ an \emph{update sequence}, and each of the $a_i (1\leq i\leq \ell)$ an \emph{update} (so that there are $\ell$ updates in the update sequence $(a_1, \ldots, a_\ell)$).
\end{remark}

Let $F_1$ be a distribution supported on odd integers such that $F_1(2j+1) = \frac{x^{2j+1}}{{(2j+1)!\sinh x}}$ for all $j\geq 0$, and let $F_2$ be a distribution supported on even integers such that $F_2(2j) = \frac{ x^{2j}}{{(2j)!\cosh x}}$ for all $j\geq 0$. For a fixed $1\leq k\leq N$, let $U_i$ be i.i.d.\ random variables distributed as $F_1$ for $i\in \{1, \ldots, n\} \setminus \{k\}$  and independently let $U_i$ be i.i.d.\ random variables distributed as $F_2$ for $i\in \{n+1, \ldots, N\} \setminus \{k\}$, and let $U_k$ be another independent random variable with distribution $F_2$ if $1\leq k\leq n$ and with distribution $F_1$ if $n+1 \leq k\leq N$. Given the values of $U_1, \ldots, U_N$, we let $(A_1, \ldots, A_{L-1}, k)\in \{1,\ldots, N\}^L$ (where $L-1 = \sum_{i=1}^N U_i$) be a sequence uniformly at random subject to $|\{1\leq j \leq L-1: A_j = i\}| = U_i$. We denote by $\mu_{k, n }$ the probability measure of the random sequence $(A_1, \ldots, A_{L-1}, k)$.

\begin{lemma}\label{lem-mu-k-n}
	For $1\leq k\leq n \leq \ell$ and any sequence $(a_1, \ldots, a_{\ell -1}, k) \in \mathcal S_{k}(n, \ell)$, we have
	\begin{equation}\label{lem-mu-k-n-eq-1}
	\mu_{k, n} ((a_1, \ldots, a_{\ell-1}, k)) = \tfrac{x^{\ell-1}}{(\ell-1)!} \tfrac{1}{( \sinh x)^{n-1}}\tfrac{1}{ (\cosh x)^{N-n+1} }\,.
	\end{equation}
	Similarly, for $n+1\leq k\leq N$ and $\ell\geq n+2$, and any sequence $(a_1, \ldots, a_{\ell -1}, k) \in \mathcal S_{k}(n, \ell)$, we have
	\begin{equation}\label{lem-mu-k-n-eq-2}
	\mu_{k, n} ((a_1, \ldots, a_{\ell-1}, k)) = \tfrac{x^{\ell-1}}{(\ell-1)!} \tfrac{1}{( \sinh x)^{n+1}}\tfrac{1}{ (\cosh x)^{N-n-1} }\,.
	\end{equation}
\end{lemma}
\begin{proof}
	We only prove the first case. Let $n_i = |\{1\leq j\leq \ell -1 : a_j = i\}|$.  Then we have
	\begin{equation}\label{eq-mu-k-n}
	\mu_{k, n} ((a_1, \ldots, a_{\ell-1}, k))=  \mu_{k, n} (U_i = n_i \mbox{ for all } 1\leq i\leq N) \cdot \tfrac{\prod_{i=1}^N n_i!}{(\ell - 1)!}\,,
	\end{equation}
	where the second term on the right hand side counts the conditional probability of sampling $(a_1, \ldots, a_{\ell-1}, k)$ given $U_i = n_i$ for all $1\leq i\leq N$. By independence of $U_i$'s, we see that
	\begin{align*}
	\mu_{k, n}(U_i = n_i \mbox{ for all } 1\leq i\leq N)  & = \prod_{i=1}^N \mu_{k, n}(U_i = n_i) = \prod_{1\leq i \neq k \leq n} F_1(n_i) \cdot \prod_{n+1\leq i\leq N} F_2(n_i) \cdot F_2(n_k) \\
	& = \prod_{1\leq i\neq k\leq n} \frac{x^{n_i}}{n_i! \sinh x}  \cdot \prod_{n+1\leq i\leq N} \frac{x^{n_i}}{n_i! \cosh x} \cdot \frac{x^{n_k}}{n_k! \cosh x}\\
	& = x^{\ell-1} \tfrac{1}{\prod_{i=1}^N n_i!}  \tfrac{1}{( \sinh x)^{n-1}}\tfrac{1}{ (\cosh x)^{N-n+1} }\,.
	\end{align*}
	Combined with \eqref{eq-mu-k-n}, this completes the proof of the first part of the lemma. The second part is similar. \qed
\end{proof}

\begin{lemma}\label{lem-M-k-n}
	We have
	\begin{equation}
	\sum_{\ell=1}^{\infty} M(n,\ell){x^\ell\over \ell!}=(\sinh x)^n(\cosh x)^{N-n}\,.
	\end{equation}
	In addition, we have
	\begin{equation}
	\begin{aligned}
	\label{lem-M-k-n2}
	\sum_{\ell=1}^{\infty} M(n,\ell){x^{\ell-1}\over (\ell-1)!} &=((\sinh x)^n (\cosh x)^{N-n})'  \\
	&=(\sinh x)^{n-1} (\cosh x)^{N-n-1} (n(\cosh x)^2+(N-n)(\sinh x)^2)\,.
	\end{aligned}
	\end{equation}
\end{lemma}

\begin{proof}
	We give a proof of the second equality. The first equality can be obtained by integrating the second equality with respect to $x$.
	
	Since $\mu_{k, n}$ is a probability measure on $\mathcal S_k(n)$, we see that
	$\sum_{\vec a \in \mathcal S_k(n)} \mu_{k, n} (\vec a)=1$. Combined with Lemma~\ref{lem-mu-k-n}, it yields that when $1\leq k\leq n$
	$$1=\sum_{\ell = n}^\infty\sum_{\vec a \in \mathcal S_k(n, \ell)} \mu_{k, n}(\vec a)  = \sum_{\ell = n}^\infty |\mathcal S_k(n, \ell)| \tfrac{x^{\ell-1}}{(\ell-1)!} \tfrac{1}{( \sinh x)^{n-1}}\tfrac{1}{ (\cosh x)^{N-n+1} }\,, $$
	and when $n+1\leq k\leq N$
	$$1=\sum_{\ell = n+2}^\infty\sum_{\vec a \in \mathcal S_k(n, \ell)} \mu_{k, n}(\vec a)  = \sum_{\ell = n+2}^\infty |\mathcal S_k(n, \ell)| \tfrac{x^{\ell-1}}{(\ell-1)!} \tfrac{1}{( \sinh x)^{n+1}}\tfrac{1}{ (\cosh x)^{N-n-1} }\,. $$
	This tells us that when $1\leq k\leq n$
	$$ \sum_{\ell = n}^\infty |\mathcal S_k(n, \ell)|\tfrac{x^{\ell-1}}{(\ell-1)!}=( \sinh x)^{n-1} (\cosh x)^{N-n+1} \,,$$
	and when $n+1\leq k\leq N$
	$$ \sum_{\ell = n+2}^\infty  |\mathcal S_k(n, \ell)|\tfrac{x^{\ell-1}}{(\ell-1)!}=( \sinh x)^{n+1} (\cosh x)^{N-n-1} \,.$$
	Summing these $N$ equalities (combined with the fact that $M(n, \ell) = |\mathcal M(n,\ell)| = |\mathcal S(n, \ell)| = \sum_{1\leq k\leq N} |\mathcal S_k(n, \ell)|$) completes the proof of \eqref{lem-M-k-n2} and hence the lemma.\qed
\end{proof}

\begin{corollary}\label{cor-first-moment}
	$\E Z_{N,x}\leq N (\sinh x)^{N-1}\cosh x$.
\end{corollary}
\begin{proof}
	Here we will derive an upper bound for $\E Z_{N, x}$ in the general (not necessarily antipodal) case. Suppose the Hamming distance between $u$ and $w$ is $n$. Let $\mathcal  M'(n,\ell)$ be the subset of self-avoiding paths in $\mathcal  M(n,\ell)$ and write $M'(n, \ell) = |\mathcal M'(n, \ell)|$. Since for each path $P\in \mathcal M'(n,\ell)$, the probability that $P$ is accessible is $\frac {x^{\ell-1}} {(\ell-1)!}$, we have
	\begin{align}
	\E Z_{N,x} &= \E \sum\limits_{\ell=1}^\infty\sum\limits_{P\in \mathcal M'(n,\ell)} 1_{P \mbox { is accessible }}= \sum\limits_{\ell=1}^\infty M'(n,\ell) \frac{x^{\ell-1}}{(\ell-1)!}
	\leq \sum\limits_{\ell=1}^\infty M(n,\ell) \frac{x^{\ell-1}}{(\ell-1)!} \nonumber\\
	&= (\sinh x)^{n-1} (\cosh x)^{N-n-1} (n(\cosh x)^2+(N-n)(\sinh x)^2)\,,\label{eq-upper-bound-first-moment}
	\end{align}
	where the last equality follows from \eqref{lem-M-k-n2}.
	In the antipodal case, substituting $n=N$ in \eqref{eq-upper-bound-first-moment} gives the desired bound.\qed
\end{proof}
\paragraph{Proof of \eqref{eq-mainthm-upper}: antipodal case}
	In this case, $\beta = 1$ so we have $f(x)=\sinh x$, $x_0=\sinh^{-1}(1)=\ln(\sqrt2+1)$, $\sinh x_0=1$ and $\cosh x_0=\sqrt2$.
	We can without loss of generality assume that $\epsilon_N\leq N^{-2/3}$ since $\P(Z_{N,x}>0)$ is increasing in $x$.  By Corollary~\ref{cor-first-moment}, we have (recall that  $x_c=x_0-{\frac 1{f'(x_0)}}{\frac {\ln N}N}=x_0-\frac {\sqrt 2}2 \frac {\ln N}N$)
	\begin{align*}
	\P(Z_{N,x_c-\epsilon_N}>0)
	&\leq\E Z_{N,x_c-\epsilon_N} \leq N(\sinh(x_c-\epsilon_N))^{N-1}\cosh (x_c-\epsilon_N)\\
	&= N(\sinh (x_0)-\cosh (x_0)(\frac {\sqrt 2}2 \frac {\ln N}N+\epsilon_N)+o(1/N))^{N-1}\cosh (x_c-\epsilon_N)\\
	&\leq N(1-\frac {\ln N}N-\sqrt2\epsilon_N+o(1/N))^{N-1}\sqrt2 \to 0 \mbox{ as }N\to \infty\,. 
	\end{align*}\qed

\begin{remark}\label{remark-3}
	Similarly we can show that for $x=x_c+\epsilon_N$ and $N\epsilon_N\to \infty$, we have $N(\sinh x)^{N-1}\cosh x=N(\sinh(x_c+\epsilon_N))^{N-1}\cosh (x_c+\epsilon_N)\to \infty$ as $N\to \infty$, and that for all $x=x_c+\epsilon_N$ such that $|N\epsilon_N|\leq \Delta$, we have $m_1(\Delta)\leq N(\sinh x)^{N-1}\cosh x\leq m_2(\Delta)$ where $m_1(\Delta), m_2(\Delta)>0$ depend only on $\Delta$.
	Combined with Lemma \ref{lem-first-moment} below, this suggests (at least in expectation) that $x_c$ is the critical value.
\end{remark}

\subsection{Proof of the lower bound}
In order to prove the lower bound, we restrict our attention to certain good paths, i.e., those with desirable properties on the growth of Hamming distances (in particular, a good path needs to be self-avoiding). We will define precisely what we mean by a good path in Definition \ref{def-good-antipodal} below. Denote by $Z_{N,x,*}$ the number of good accessible paths. Crucially, we demonstrate that with our definition of good paths, we have $\E Z_{N, x, *} \asymp \E Z_{N, x}$ and $\E Z_{N, x, *}^2 \asymp (\E Z_{N, x, *})^2$ (where $\asymp$ means that the left and right hand sides are within a constant multiplicative factor) as long as $x=x_c+\epsilon_N$ ($N \epsilon_N  \to \infty$) and $x$ stays in a fixed neighborhood of $x_0$. Thus, an application of the second moment method already yields the existence of an accessible path with probability bounded away from 0. Finally, we use the augmenting method as employed in \cite{HM14} to deduce the existence of an accessible path with probability tending to 1 as $N \to \infty$.

Recall that $x_0=\sinh^{-1}(1)=\ln(\sqrt2+1)\approx 0.88137$. Let $\alpha =x_0\coth x_0\approx 1.24645$.

For any $0<\epsilon<1$, we set $\epsilon_1, \epsilon_2 \mbox{ and } \epsilon_3$ throughout the rest of the paper as
\begin{equation}\label{eq-def-epsilons}
\epsilon_1 = \epsilon^{1/2}\,, \epsilon_2 = \epsilon^{1/4} \mbox{ and } \epsilon_3 = \epsilon^{1/8}\,.
\end{equation}
We will fix $\epsilon$ to be a certain sufficiently small number later.
%
%Let us say a few words about these ``infinitesimals'' $\iota, \epsilon_1, \epsilon_2, \epsilon_3$ and $\epsilon$. They are not really infinitesimals because throughout the paper we only need to fix each of them to be a certain sufficiently small number (and thanks to \eqref{eq-def-epsilons} we only need to fix $\epsilon$). However, to guarantee a viable choice of these numbers, we need $\iota, \epsilon, \epsilon_1, \epsilon_2, \epsilon_3$ to be decreasing in terms of the orders of the infinitesimals, hence our definition \eqref{eq-def-epsilons}. Our specific choices of the relations between $\iota, \epsilon_1, \epsilon_2, \epsilon_3$ and $\epsilon$ though are rather arbitrary.
For $u, v \in H_N$, we denote by $H(u, v)$ the Hamming distance between $u$ and $v$.
\begin{defn} \label{def-good-antipodal}
	Let $\epsilon>0$ be a sufficiently small fixed number to be selected.
	We say a path (or the associated update sequence) $v_0=\vec{0}, v_1, \ldots, v_{L-1},v_L=\vec{1}$ is \emph{good} if  $L\in[\alpha(1-\epsilon)N,\alpha(1+\epsilon)N]$ and the following holds:
	
	\noindent
	$H(v_i,v_j)=|i-j|, \text{ if } |i-j|=1,2,3;\\
	H(v_i,v_j)=|i-j| \text{ or } |i-j|-2,  \text{ if } 4\leq|i-j|\leq N^{\frac 1 5};\\
	H(v_i,v_j)\leq (1/2+\epsilon_1)N, \text{ if } N^{\frac 1 5}\leq |i-j|\leq\alpha(1/2+\epsilon)N;\\
	H(v_i,v_j)>(1/2+\epsilon_1)N, \text{ if } |i-j| > \alpha(1/2+\epsilon_2)N;\\
	H(v_i,v_j)\geq {\frac {|i-j|}{\alpha+\epsilon_3}}, \text{ if } N^{\frac 1 5}\leq|i-j|\leq \alpha(1/2+\epsilon_2)N.$
\end{defn}
It is clear from the definition that a good path is self-avoiding.

\begin{lemma}\label{lem-first-moment}
	For any sufficiently small but fixed number $\epsilon>0$, there exist $C_1>0$ and an integer $N'>0$ which both depend only on $\epsilon$, such that for all $|x - x_0| \leq \epsilon^2$ and $N>N'$ we have
	\begin{equation}\label{toshow}
	\E Z_{N,x,*}\geq C_1N\sinh^{N-1}x\cosh x\,.
	\end{equation}
\end{lemma}
\begin{proof}
	We keep all the definitions and notations in the previous subsection \ref {subsec:Proof of the upper bound}. Since we are working in the antipodal case where $\beta=1$, we have substituted $n$ by $N$ in the following without further notice. Recall that as stated in Definition \ref{def-good-antipodal}, an update sequence is good if its corresponding path is good. For each $1\leq k\leq N$, we let $\mathcal S_{k, *}(N) \subseteq \mathcal S_k(N)$ contain all the good sequences ending in $k$, and let $\mathcal M_{k, *}(N)$ be the collection of the corresponding good paths. We claim that in order to show \eqref{toshow}, it suffices to show that for each $1\leq k\leq N$
	\begin{equation}\label{eq-good}
	\mu_{k, N}(\mathcal S_{k, *}(N)) \geq C_1\,.
	\end{equation}
	Indeed, summing equation \eqref{lem-mu-k-n-eq-1} over all $(a_1, \ldots, a_{\ell -1}, k) \in \mathcal S_{k, *}(N)$ gives that
	$$\mu_{k, N}(\mathcal S_{k, *}(N))=\tfrac{1}{( \sinh x)^{N-1}\cosh x}\sum_{\substack {P\in \mathcal M_{k, *}(N) \\ P \mbox{ is of length }\ell}}\tfrac{x^{\ell-1}}{(\ell-1)!}=\tfrac{1}{( \sinh x)^{N-1}\cosh x}\sum_{P\in \mathcal M_{k, *}(N)} \P (P \mbox { is accessible})\,,$$
	where the last equality is because any good path is necessarily self-avoiding. If \eqref{eq-good} holds true, then summing the above equation over $1\leq k\leq N$ yields \eqref{toshow}.
	
	For ease of elaboration we make a slight modification to \eqref{eq-good}, that is, we will show instead that
	\begin{equation}\label{eq-good-1}
	\tilde \mu_{N}(\mathcal S_{*}(N)) \geq \tilde {C_1}\,,
	\end{equation}
	where $\tilde \mu_{N}$ differs from $\mu_{k,N}$ in that we also let $U_k$ be chosen according to $F_1$ instead of $F_2$ (in other words, for each $1\leq i\leq N$, the $U_i$'s are now i.i.d.\ random variables distributed as $F_1$), and consider the random sequence $(A_1, \ldots, A_{L-1})$ instead of $(A_1, \ldots, A_{L-1},k)$. See also Case 1 below for the definition of $\tilde \mu_{N,\beta}$, the generalization of $\tilde \mu_{N}$ to general $\beta$; we use $S_{*}(N)$ to denote the collection of all the good sequences (not necessarily ending in $k$).
	
	There are a number of ways to justify our replacement of \eqref{eq-good} by \eqref{eq-good-1}. For example, one may argue that if $\tilde \mu_{N-1}(\mathcal S_{*}(N-1)) \geq \tilde {C_1}$ holds, then (possibly with a slight change of $N^{\frac 1 5}, \epsilon, \epsilon_1,\epsilon_2 \mbox{ and } \epsilon_3$ in the definition of good paths) $\mu_{k,N}(\mathcal S_{k,*}(N))=\mu_{N,N}(\mathcal S_{N,*}(N)) \geq {\frac 1 {\cosh x}}\tilde {C_1}$ holds, since
	\begin{eqnarray*}
		\mu_{N,N}(\mathcal S_{N,*}(N))&\geq &\mu_{N,N}(\{(A_1, \ldots, A_{L-1},N): U_N=0, (A_1, \ldots, A_{L-1})\in S_*(N-1)\})\\
		&=& {\frac 1 {\cosh x}} \tilde \mu_{N-1}(\mathcal S_{*}(N-1))\,.
	\end{eqnarray*}
	
	In the rest of the proof, $\P$ and $\E$ refer to $\tilde \mu_{N}$ unless otherwise specified. Note that $\P$ depends on both $x$ and $N$. Under this probability space (or the more general $\tilde \mu_{N,\beta}$), we say an event $\mathcal {E}_N$ happens with probability tending to 1 as $N\to \infty$ (or with high probability for brevity) if $1-\P(\mathcal {E}_N) \leq p(\epsilon, N)$ where $p(\epsilon, N)>0$ only depends on $\epsilon$ and $N$, and (when $\epsilon$ is fixed) goes to 0 as $N\to \infty$. Similarly, we say a quantity (possibly random) $Q_N$ is $o(1)$ if $|Q_N|\leq q_N$ where $q_N>0$ is fixed, only depends on $N$ and goes to 0 as $N\to\infty$.
	
	By a simple calculation, for $U\sim F_1$, we have $\E U=x\coth x$, and $\var U$ is bounded by an absolute constant (since $|x - x_0| \leq \epsilon^2$). Therefore it is immediate from, say, Chebyshev's inequality (as used in proving the weak law of large numbers) that with probability tending to 1 as $N\to \infty$ we have $L\in[\alpha(1-\epsilon)N,\alpha(1+\epsilon)N]$ (recall that $\alpha =x_0\coth x_0$). It now remains to consider the requirements on Hamming distances in the definition of good paths, for which purpose we split into three cases as follows.
	
	\bigskip
	\noindent {\bf Case 1:  $H(v_i,v_j)=|i-j|, \text{ if } |i-j|=1,2,3.$}
	
	We show that this requirement can be satisfied by a sequence generated from $\tilde \mu_{N}$ with probability bounded from below by a constant. We prove the following statement \eqref{eq-Case-1} for general $\beta$. 
	
	Fix a $\beta\in (0,1]$. For $i\in \{1, \ldots, \beta N\}$, let $U_i$ be i.i.d.\ random variables distributed as $F_1$, and independently for $i\in \{\beta N + 1, \ldots, N\}$, let $U_i$ be i.i.d.\ random variables distributed as $F_2$. Given the values of $U_1, \ldots, U_N$, we let $(A_1, \ldots, A_L)$ (where $L = \sum_{i=1}^N U_i$) be a sequence uniformly at random subject to $|\{1\leq j \leq L: A_j = i\}| = U_i$. Let $\tilde \mu_{N,\beta}$ be the probability measure of the random sequence $(A_1, \ldots, A_L)$ thus obtained.
	
	For convenience we set $A_{i+L}=A_i$ for $i\geq 1$. Let
	\begin{equation}\label{def-I_i-N_i}
	\begin{aligned}
	& I_i=1_{\{A_i = A_{i+1}\}} \mbox { and } \mathcal N_i=\{i,i+1\}, &\mbox{ if $i = 1, 2, \cdots, L$};\\
	& I_i=1_{\{A_{i-L} = A_{i+2-L}\}} \mbox { and } \mathcal N_i=\{i-L,i+2-L\}, &\mbox{ if $i = L+1, L+2, \cdots, 2L$}.
	\end{aligned}
	\end{equation}
	
	Let $x_0$ be given as in Theorem \ref{thm-main}, and let $\gamma = \beta x_0\coth x_0+(1-\beta)x_0\tanh x_0$. For any $\epsilon>0$, there exists a constant $c^*>0 $ and an integer $N'>0$ which both depend only on $\epsilon$, such that for all $|x - x_0| \leq \epsilon^2$ and $N>N'$ we have
	\begin{equation}\label{eq-Case-1}
	\tilde \mu_{N,\beta}(\sum_{i=1}^{2L}I_i=0)\geq c^*\,.
	\end{equation}
	
	\begin{remark}
		In fact, as can be seen from our proof, $x_0$ could be any fixed positive number (not necessarily given by Theorem \ref{thm-main}). Moreover, we have $c^*\to e^{-\tfrac{2x_0^2}{\gamma}}$ as $\epsilon\to0$, and if $x\to x_0$ as $N\to\infty$, then $\sum_{i=1}^{2L}I_i$ converges to the Poisson distribution with mean $\tfrac{2x_0^2}{\gamma}$ as $N\to\infty$. However, we don't need any of these facts.
	\end{remark}
	
\paragraph{Proof of \eqref{eq-Case-1}}
		In this proof, $\P$ and $\E$ refer to $\tilde \mu_{N,\beta}$. Let
		$$D_j:=|\{1\leq i\leq N: U_i = j\}|$$
		for $j \in \mathbb N$ and
		$$\Lambda:= L^{-1} \sum_{j=2}^\infty D_j j (j-1).$$
		
		By a simple calculation, for $U\sim F_1$, we have $\E U=x\coth x$ and $\E U(U-1)=x^2$, and the variances of $U$ and $U(U-1)$ are both bounded by an absolute constant, as long as $x$ stays in a fixed neighborhood of $x_0$. Similarly, for $U\sim F_2$, we have $\E U=x\tanh x$ and $\E U(U-1)=x^2$, and the variances of $U$ and $U(U-1)$ are both bounded by an absolute constant. By Chebyshev's inequality, we have with probability tending to 1 as $N\to \infty$,
		\begin{equation}\label{L}
		L = \sum_{i=1}^N U_i \in [\gamma (1-\epsilon) N,\gamma (1+\epsilon) N]
		\end{equation}
		and
		\begin{equation}\label{Dj}
		\sum_{j=2}^\infty D_j j (j-1) = \sum_{i=1}^N U_i(U_i-1) \in [x_0^2 (1-\epsilon) N, x_0^2 (1+\epsilon)N]\,.
		\end{equation}
		\eqref{L} and \eqref{Dj} combined give
		\begin{equation*}
		\Lambda\in [(1-3\epsilon) \tfrac{x_0^2} {\gamma}, (1+3\epsilon) \tfrac{x_0^2} {\gamma}]\,.
		\end{equation*}
		By the uniform convergence of $\sum_{k=1}^K(-1)^{k+1}\frac {(2 \Lambda)^k}{k!}$ to $1-e^{-2\Lambda}$ on $[(1-3\epsilon) \tfrac{x_0^2} {\gamma}, (1+3\epsilon) \tfrac{x_0^2} {\gamma}]$, there exists a finite odd number $K$ and $0<c^{**}<1$ ($c^{**}$ may depend on $K$ and $\epsilon$) such that for all $\Lambda \in [(1-3\epsilon) \tfrac{x_0^2} {\gamma}, (1+3\epsilon) \tfrac{x_0^2} {\gamma}]$, we have
		\begin{equation}\label{eq-def-K}
		\sum_{k=1}^K(-1)^{k+1}\frac {(2 \Lambda)^k}{k!}<c^{**}\,.
		\end{equation}
		Again, by Chebyshev's inequality, we have with probability tending to 1 as $N\to \infty$,
		\begin{equation}\label{2Kmoment}
		\sum_{j=0}^\infty D_j j^{2k} = \sum_{i=1}^N U_i^{2k}  \leq C_K N \,, \mbox{ for all } 1\leq k\leq K
		\end{equation}
		where $C_K>0$ is a constant which only depends on $K$. Also, by a rather loose bound on $\P(U_i\geq 10 \log N)$ (directly from the definition of $U_i$), we have with probability tending to 1 as $N\to \infty$,
		\begin{equation}\label{10logN}
		\max_{1\leq i \leq N} U_i \leq 10 \log N\,.
		\end{equation}
		We will assume \eqref{L}, \eqref{Dj}, \eqref{2Kmoment} and \eqref{10logN} without mention in what follows.

		Write $\mathcal{F} = \sigma(U_1, U_2, \ldots, U_N)$. By Bonferroni's inequalities \cite{Bon36}, we have
		\begin{equation}\label{label-Bonferroni}
		\P\big(\sum_{i=1}^{2L} I_i\geq 1\mid \mathcal{F}\big)\leq \sum_{k=1}^K (-1)^{k+1}\sum_{1 \leq i_1 < i_2 < \cdots < i_k \leq 2L}\P(I_{i_1}=1,I_{i_2}=1,\cdots,I_{i_k}=1\mid \mathcal{F})\,.
		\end{equation}
		In order to prove \eqref{eq-Case-1}, it suffices to show that each summand (of $\sum_{k=1}^K$) on the right hand side of \eqref{label-Bonferroni} is asymptotic to the corresponding summand on the left hand side of \eqref{eq-def-K}. That is to say, we want to show that for each $1\leq k\leq K$,
		\begin{equation} \label{eq-each-k}
		\sum_{1 \leq i_1 < i_2 < \cdots < i_k \leq 2L}\P(I_{i_1}=1,I_{i_2}=1,\cdots,I_{i_k}=1\mid \mathcal{F})-\frac {(2\Lambda)^k}{k!}=o(1)\,.
		\end{equation}
		For this purpose, we will split $\sum\limits_{1 \leq i_1 < i_2 < \cdots < i_k \leq 2L}\P(I_{i_1}=1,I_{i_2}=1,\cdots,I_{i_k}=1\mid \mathcal{F})$ into two parts according to whether or not any $A_i$ is involved in the definition of more than one $I_{i_j}$'s ($1 \leq j \leq k$). More precisely, for a pair of integers $(i_j, i_{j'})$ (or equivalently $(I_{i_j}, I_{i_{j'}})$) where $i_j\neq i_{j'}$   we say it is \emph{intersecting} if $\mathcal N_{i_j}\cap \mathcal N_{i_j'}\neq \emptyset$ (see \eqref{def-I_i-N_i} for the definition of $\mathcal N_{i}$). Let $\mathcal I^{k,1}$ ($\mathcal I^{k,2}$) denote the set of all sequences $(i_1, i_2, \cdots, i_k)$ such that $1 \leq i_1 < i_2 < \cdots < i_k \leq 2L$ and it contains no (at least 1) intersecting pair, respectively. We can write
		$$\sum_{1 \leq i_1 < i_2 < \cdots < i_k \leq 2L}\P(I_{i_1}=1,I_{i_2}=1,\cdots,I_{i_k}=1\mid \mathcal{F})  = \mathcal J_1 + \mathcal J_2$$
		where
		\begin{eqnarray*}\label{eq-J-1-2}
			\mathcal J_1 = \sum_{\mathcal I^{k,1}}\P(I_{i_1}=1,I_{i_2}=1,\cdots,I_{i_k}=1\mid \mathcal{F}) \mbox{ and } \mathcal J_2 = \sum_{\mathcal I^{k,2}}\P(I_{i_1}=1,I_{i_2}=1,\cdots,I_{i_k}=1\mid \mathcal{F}).
		\end{eqnarray*}
		
		We first bound the term $\mathcal J_1$.  For any $(i_1, i_2, \cdots, i_k)\in \mathcal I^{k,1}$, the neighborhoods $\mathcal N_{i_1}, \mathcal N_{i_2},\cdots,\mathcal N_{i_k}$ are disjoint by definition. Now given $\mathcal{F}$, for each $r=1,\ldots,k$, there are at most $\sum_{j=2}^\infty D_j\cdot j\cdot(j-1)$ ways of choosing two matching updates for the two slots in $\mathcal N_{i_r}$, and there are at most $(L - 2k)!$ ways of arranging the remaining $(L-2k)$ updates, therefore we have
		\begin{eqnarray}\label{J_1-upper-bound}
		\P(I_{i_1} = 1, I_{i_2} = 1,\cdots, I_{i_k} = 1 \mid \mathcal{F}) &\leq & \frac{(L - 2k)!}{L!}\Big(\sum_{j=2}^\infty D_j\cdot j\cdot(j-1)\Big)^k\\
		&=& (\frac 1 L)^k(1+o(1))\Lambda^k\nonumber.
		\end{eqnarray}
		Combined with the simple fact that $|\mathcal I^{k, 1}| \leq (2L)^k/k!$, this gives that
		$\mathcal J_1 \leq (2 \Lambda)^k(1+o(1))/k!$.
		On the other hand, by a similar reasoning
		\begin{eqnarray*}
			\P(I_{i_1}=1,I_{i_2}=1,\cdots,I_{i_k}=1 \mid \mathcal{F}) &\geq & \frac{(L - 2k)!}{L!}\prod_{1 \leq r \leq k}\big( \sum_{j=2}^{10\log N} (D_j - (r-1))\cdot j\cdot(j-1) \big)\\
			& \geq & (\frac 1 L)^k(1+o(1))(\Lambda+o(1))^k.
		\end{eqnarray*}
		Moreover, we have $|\mathcal I^{k,1}| \geq (1+o(1)) (2L)^k / k!$ since $|\mathcal I^{k,1}| \geq  \prod\limits_{1 \leq r \leq k}(2L - 7(r-1)) / k!$ (each $\mathcal N_{i}$ intersects 6 other $\mathcal N_{i}$'s). Hence, we obtain that
		$\mathcal J_1 \geq (2 \Lambda)^k(1+o(1))/k!$.
		Altogether, we get
		\begin{equation}\label{eq-I-k-1}
		\mathcal J_1=  (2 \Lambda)^k(1+o(1))/k!\,.
		\end{equation}
		
		It remains to control $\mathcal J_2$. For any $(i_1, i_2, \cdots, i_k)\in \mathcal I^{k,2}$, denote by $\mathcal E_{i_1, \ldots, i_k} = \{(A_1, A_2,\ldots, A_L):I_{i_1}=1,I_{i_2}=1,\cdots, I_{i_k}=1\}$. Observe that $I_{i_1}=1,I_{i_2}=1,\cdots, I_{i_k}=1$ (the criteria for $\mathcal E_{i_1, \ldots, i_k}$)  can be rewritten (or simplified) uniquely as a set of equalities
		\begin{align*}
		& A_{j_1}=A_{j_1+n_{1,1}}=A_{j_1+n_{1,1}+n_{1,2}}=\cdots=A_{j_1+n_{1,1}+n_{1,2}+\cdots+n_{1,a_1-1}}\\
		& A_{j_2}=A_{j_2+n_{2,1}}=A_{j_2+n_{2,1}+n_{2,2}}=\cdots=A_{j_2+n_{2,1}+n_{2,2}+\cdots+n_{2,a_2-1}}\\
		&\cdots\\
		& A_{j_\ell}=A_{j_\ell + n_{\ell,1}} = A_{j_\ell + n_{\ell,1}+n_{\ell,2}}=\cdots=A_{j_\ell + n_{\ell,1} + n_{\ell,2} + \cdots + n_{\ell, a_\ell-1}}
		\end{align*}
		where $n_{1,1},\ldots,n_{1,a_1-1}, n_{2,1}, \ldots, n_{2,a_2-1}, \ldots, n_{\ell,1}, \ldots, n_{\ell,a_\ell-1}$ are either 1 or 2, $a_1,a_2,\ldots,a_\ell$ are integers $\geq 2$ and $a_1 + a_2 + \cdots + a_\ell \leq 2k$ (in particular each $a_i$ is $\leq 2k$). Also, since $(i_1, i_2, \cdots, i_k)\in \mathcal I^{k,2}$, i.e. there is at least one intersecting pair in $I_{i_1},\cdots, I_{i_k}$, at least one of the $a_1,a_2,\ldots, a_\ell$ must be strictly larger than 2, so that $a_1 + a_2 + \cdots + a_\ell > 2\ell$. Denote by $\mathcal A$ the preceding set of equalities (so $\mathcal A$ can also be viewed as an event). By a rather loose bound, $|\{(i_1, \ldots, i_k): \mathcal E_{i_1, \ldots, i_k} = \mathcal A\}| \leq (a_1 + a_2 + \cdots + a_\ell)^{2k} \leq (2k)^{2k}$. Therefore we have
		\begin{eqnarray}\label{eq-E-A}
		\sum_{\mathcal I^{k,2}}\P(\mathcal E_{i_1, \ldots, i_k}\mid \mathcal{F}) &\leq & (2k)^{2k}\sum_{\ell}\sum_{\mathcal{D}_1}\sum_{\mathcal{D}_2}\sum_{\mathcal{D}_3}\P( \mathcal A\mid \mathcal{F}),
		\end{eqnarray}
		where $\mathcal{D}_1, \mathcal{D}_2, \mathcal{D}_3$ respectively denote the collections of all valid choices of $(a_1, a_2, \ldots, a_{\ell})$, \\$(n_{1,1},\ldots,n_{1,a_1-1}, n_{2,1}, \ldots, n_{2,a_2-1}, \ldots, n_{\ell,1}, \ldots, n_{\ell,a_\ell-1})$ and $(j_1,j_2,\ldots,j_\ell)$.
		Now similar to \eqref{J_1-upper-bound}, we have
		\begin{align*}
		\P(\mathcal A\mid \mathcal{F}) \leq \frac{(L-(a_1+a_2+\cdots+a_\ell))!}{L!}\prod_{r = 1}^\ell\Big(\sum_{i=a_r}^\infty D_i\cdot i\cdot (i-1)\cdots(i - a_r + 1)\Big)\,.
		\end{align*}
		Therefore, by \eqref{2Kmoment} we have
		\begin{align}\label{eq-prob-A}
		\sum_{\mathcal{D}_3}\P(\mathcal A\mid \mathcal{F})  \leq C'_K N^{2\ell-(a_1+a_2+\cdots+a_\ell)} \leq C'_K/N\,,
		\end{align}
		where $C'_K$ is another constant depending on $K$, and the second inequality follows from the fact that $a_1 + a_2 + \cdots + a_\ell > 2\ell$. Since $|\mathcal D_1|$, $|\mathcal{D}_2|$ and $\ell$ are all bounded by a number that depends only on $K$, we combine \eqref{eq-E-A} and \eqref{eq-prob-A} and obtain
		$$\sum_{\mathcal I^{k,2}}\P(\mathcal E_{i_1, \ldots, i_k}\mid \mathcal{F}) \leq C^*_K/N\,,$$
		where $C^*_K>0$ depends only on $K$. Combined with \eqref{eq-I-k-1}, this yields \eqref{eq-each-k} and therefore \eqref{eq-Case-1}.\qed

	\bigskip
	\noindent {\bf Case 2 : $H(v_i,v_j)=|i-j| \text{ or } |i-j|-2,  \text{ if } 4\leq|i-j|\leq N^{\frac 1 5}.$}
	
	We show that this requirement is satisfied by a sequence generated from $\tilde \mu_{N}$ with probability tending to 1 as $N\to \infty$. Denote by $W_k$ the event that in some $k$ consecutive updates there are at least two coordinates such that all of them occur at least twice. It suffices to show that $W_{N^{1/5}}$ happens with probability tending to 0 as $N\to\infty$. Given $\mathcal{F} = \sigma(U_1, U_2, \ldots, U_N)$, the conditional probability that the coordinates 1 and 2 both occur at least twice in the first $k$ updates is less than ${U_1\choose 2}({\frac k L})^2{U_2\choose 2}({\frac k L})^2$, by a union bound. Therefore,
	\begin{equation}\label{eq-Case-2}
	\mathbb{P}(W_k)  = \E (\mathbb P (W_k \mid \mathcal F)) \leq \sum_{1\leq i<j\leq N} \E \Big({U_i\choose 2}({\frac k L})^2{U_j\choose 2}({\frac k L})^2 L\Big) \leq \frac {C' k^4} N = o(1)
	\end{equation}
	for $k=N^{1/5}$ (here $C'$ is an absolute constant).
	
	\bigskip
	\noindent{\bf Case 3: 	\\
		$H(v_i,v_j)\leq (1/2+\epsilon_1)N, \text{ if } N^{\frac 1 5}\leq |i-j|\leq\alpha(1/2+\epsilon)N;\\
		H(v_i,v_j)>(1/2+\epsilon_1)N, \text{ if } |i-j| > \alpha(1/2+\epsilon_2)N;\\
		H(v_i,v_j)\geq {\frac {|i-j|}{\alpha+\epsilon_3}}, \text{ if } N^{\frac 1 5}\leq|i-j|\leq \alpha(1/2+\epsilon_2)N.$}
	
	We show that these three requirements are satisfied by a sequence generated from $\tilde \mu_{N}$ with probability tending to 1 as $N\to \infty$.
	Let $\mathcal R$ be the collection of all sequences satisfying these three requirements.
	
	Before we proceed, let us first give a hint on why this may be true (i.e. what these three requirements are trying to say). For $t\in[0,1]$, we define
	\begin{equation}\label{def-g(t)}
	g(t):= \frac {\sinh(x_0t)\cosh(x_0(1-t))}{\sinh x_0}=\sinh(x_0t)\cosh(x_0(1-t))\,.
	\end{equation}
	Vaguely (and roughly) speaking, $g(t)N$ is the ``expected Hamming distance traveled by a path in time $t$'' (if the whole path uses a unit time). We will make this precise below. For a derivation of the formula \eqref{def-g(t)}, see equation \eqref{eq-prob-odd}. 
	By plotting $g(t)$ (or an easy calculus), one can easily see that
	\begin{itemize}
		\item $g(t) \leq \frac 1 2, \mbox{ if } 0 \leq t\leq{\frac 1 2}$
		\item $g(t) \geq \frac 1 2, \mbox{ if } {\frac 1 2} \leq t\leq 1$
		\item $g(t)\geq t, \mbox{ if } 0 \leq t\leq{\frac 1 2}$
	\end{itemize}
	which correspond to the three requirements, respectively. We now carry out the idea above fully and rigorously as follows.
	
	We will consider the following continuous version of $\tilde \mu_{N}$, namely $\hat \mu_N$: 
	As in $\tilde \mu_{N}$, we first let $U_i, 1\leq i\leq N$ be i.i.d.\ random variables distributed as $F_1$. Now given the values of $U_1, \ldots, U_N$, we denote $\mathcal L= \{(i,j): 1\leq i\leq N, 1\leq j\leq U_i\}$ and $L=|\mathcal L|=\sum_{i=1}^N U_i$, and let $\{r_{i,j}: (i,j)\in \mathcal L\}$ be $L$ i.i.d.\ uniform $[0,1]$ random variables. Let $\hat \mu_N$ be the underlying probability measure $F_1^N\times U[0,1]^\infty$.
	
	For each $1\leq i\leq N$, we attach the label ``$i$'' to each real number $r_{i,j}, (i,j)\in \mathcal L$. Since almost surely under $\hat \mu_N$, $L$ is finite and $r_{i,j}$'s are distinct, we can (without ambiguity) let $r_1<r_2<\cdots<r_L$ be the reordering of the reals $r_{i,j}, (i,j)\in \mathcal L$ in increasing order, and for $1\leq \ell \leq L$ let $\hat A_\ell$ be the unique label of $r_\ell$. We have thus formed a random integer sequence $(\hat A_1, \ldots, \hat A_L)$ under $\hat \mu_N$.
	
	It is clear that $(\hat A_1, \ldots, \hat A_L)$ under $\hat \mu_N$ has the same distribution as $(A_1, \ldots, A_L)$ under $\tilde \mu_{N}$, i.e., for any integer sequence $(a_1,\ldots, a_L)$, we have
	$$\hat \mu_N((\hat A_1, \ldots, \hat A_L)=(a_1,\ldots, a_L))=\tilde \mu_{N}((A_1, \ldots, A_L)=(a_1,\ldots, a_L))\,.$$
	Therefore
	\begin{equation} \label {eq-continuous-model-1}
	\hat \mu_N((\hat A_1, \ldots, \hat A_L)\in \mathcal R)= \tilde \mu_{N}((A_1, \ldots, A_L)\in \mathcal R)\,.
	\end{equation}
	
	For any interval $I\subseteq [0,1]$ and any $1\leq i\leq N$, we let $N_{I,i}$ be the number of labels ``$i$'' in $I$, i.e., $N_{I,i}=|\{1\leq j\leq U_i:r_{i,j}\in I\}|$. Let
	$$T_I=\sum_{i=1}^N N_{I,i}=|\{(i,j)\in \mathcal L:r_{i,j}\in I\}|$$
	be the total number of labels in $I$ and
	$$O_I=\sum\limits_{i=1}^N 1_{\{N_{I,i} \mbox { is an odd number}\}}$$
	count all the $i$'s ($1\leq i\leq N$) that appear an odd number of times as a label in $I$. Let $\hat {\mathcal {R}}$ be the following event: for all intervals $I\subseteq [0,1]$, we have
	\begin{align*}
	& O_I\leq (1/2+\epsilon_1)N, \text{ if } N^{\frac 1 5}\leq T_I\leq\alpha(1/2+\epsilon)N;\\
	& O_I>(1/2+\epsilon_1)N, \text{ if } T_I > \alpha(1/2+\epsilon_2)N;\\
	& O_I\geq {\frac {T_I}{\alpha+\epsilon_3}}, \text{ if } N^{\frac 1 5}\leq T_I \leq \alpha(1/2+\epsilon_2)N.
	\end{align*}
	We see that
	\begin{equation} \label {eq-continuous-model-2}
	\hat \mu_N(\hat {\mathcal {R}})= \hat \mu_N((\hat A_1, \ldots, \hat A_L)\in \mathcal R)\,.
	\end{equation}
	
	In light of equalities \eqref{eq-continuous-model-1} and \eqref{eq-continuous-model-2}, it suffices to show that under $\hat \mu_N$, $\hat {\mathcal {R}}$ happens with probability tending to 1 as $N\to \infty$. In the following $\P$ and $\E$ refer to $\hat \mu_N$. To this end, our strategy is to first show that with high probability, for all intervals $I\subseteq [0,1]$ such that $|I|\geq N^{-5 / 6}$, both $T_I$ and $O_I$ are concentrated around their means respectively.
	
	For any interval $I\subseteq [0,1]$ of length $t$, conditioning on $T_{[0,1]}=L$, $T_I$ is the sum of $L$ i.i.d.\ Bernoulli random variables with mean $t$, thus by Chernoff's bound \cite{Ch52},
	\begin{equation}
	\label{conc1}
	\P(|T_I-Lt|\geq \epsilon Lt|L)\leq 2\exp(-\epsilon^2 Lt/3).
	\end{equation}
	
	For $O_I$, by definition $O_I=\sum\limits_{i=1}^N 1_{\{N_{I,i} \mbox { is an odd number}\}}$ where $1_{\{N_{I,i} \mbox { is an odd number}\}}$ for $1\leq i\leq N$ are $N$ i.i.d.\ Bernoulli random variables with mean $p_I=\P(N_{I,1} \mbox { is an odd number})$. We can compute $p_I$ as follows:
	\begin{eqnarray}
	p_I & = & \P(N_{I,1} \mbox { is an odd number}) \nonumber\\
	& = & \sum_{i=0}^{\infty}{\frac {x^{2i+1}}{(2i+1)!\sinh x}} \sum_{j=0}^i{{2i+1}\choose {2j+1}}t^{2j+1}(1-t)^{2i-2j} \nonumber\\
	& = & \frac 1 {\sinh x}\Big(\sum_{j=0}^\infty \frac {(xt)^{2j+1}}{(2j+1)!}\Big)\Big(\sum_{i-j=0}^\infty\frac {(x(1-t))^{2(i-j)}}{(2i-2j)!}\Big)\nonumber \\
	& = & \frac {\sinh (xt)\cosh(x(1-t))} {\sinh x}.\label{eq-prob-odd}
	\end{eqnarray}
	By Chernoff's bound again, we have
	\begin{equation}
	\label{conc2}
	\P(|O_I-\E O_I|\geq 3\epsilon \E O_I)\leq 2\exp\Big(-3\epsilon^2 N{\frac {\sinh (xt)\cosh(x(1-t))}{\sinh x}}\Big).
	\end{equation}
	
	Now let us divide $[0,1]$ into $N$ non-overlapping intervals of equal length $1/N$. We say an interval is \emph{integral} if it is of the form $[n_1 / N, n_2 / N]$, where $n_1, n_2\in \mathbb N, 0 \leq n_1 < n_2 \leq N$ and $n_2 - n_1 \geq N^{1/6}$ (so that its length is at least $N^{-5 / 6}$). Denote by $E_L$ the event $\{\frac L {(x\coth x) N} \in [1-\epsilon, 1+\epsilon]\}$. Since on $E_L$, $Lt  \geq c N^{1/6}$ when $t\geq N^{-5 / 6}$ for a constant $c>0$, we can apply \eqref{conc1} and a union bound over all integral intervals to obtain that
	$$\P\Big(\max_{I \text{ is integral }} \mid T_I - Lt_I \mid \geq \epsilon Lt_I \mid L\Big)\leq 2(N+1)^2\exp(-\epsilon^2 c N^{1/6}/3), \mbox { on } E_L.$$
	Since $\E T_I=\E Lt_I=(x\coth x) N t_I$ and therefore $Lt_I\in [(1-\epsilon)\E T_I,(1+\epsilon)\E T_I]$ on $E_L$, we have
	$$\P\Big(\max_{I \text{ is integral }} \mid T_I - \E T_I \mid \geq 3\epsilon \E (T_I) \mid L\Big) \leq 2(N+1)^2\exp(-\epsilon^2 c N^{1/6}/3), \mbox { on } E_L.$$
	Since $E_L$ happens with probability tending to 1 as $N\to \infty$, we thus have that $\mathcal E_T$ happens with probability tending to 1 as $N\to \infty$, where
	$$ \mathcal E_T = \bigcap\limits_{I \mbox{ is integral }} \{T_I \in [(1-3\epsilon)\E T_I, (1+3\epsilon) \E T_I]\}\,.$$ 
	
	From \eqref{conc2}, since $\sinh x \geq x$ for $x \geq 0$, we have $N p_I \geq c N^{1/6}$ when $t \geq N^{-5/6}$ for a constant $c>0$, we can simply do a union bound over all integral $I$ and deduce that $\mathcal E_O$ happens with probability tending to 1 as $N\to \infty$, where
	$$\mathcal E_O =  \bigcap\limits_{I \mbox{ is integral }} \{O_I \in [(1-3\epsilon)\E O_I, (1+3\epsilon)\E O_I]\}\,.$$
	So we may assume without loss that both $\mathcal E_T$ and $\mathcal E_O$ occur, i.e., both $T_I$ and $O_I$ are within $[1-3\epsilon, 1+3\epsilon]$ times their respective means for any integral interval $I$.
	
	We will now argue that with high probability, both $T_I$ and $O_I$ are within $[1-4\epsilon, 1+4\epsilon]$ times their respective means for {\emph {any}} interval $I$ such that $|I|\geq N^{-5 / 6}$. For convenience we call any interval $[i/N, (i+1)/N]$ (where $0\leq i \leq N-1$) a small interval. For any small interval, the probability that there are at least $100 \log N$ labels in it is bounded by $\E\binom{L}{100 \log N}/ N^{100 \log N}$, which is at most $1/N^2$ for all large $N$. Therefore by applying a union bound over all $N$ small intervals, we have that the probability that some small interval contains at least $100\log N$ labels is $o(1)$. Without loss of generality we assume this event does not occur (i.e., any small interval contains less than $100\log N$ labels) in what follows. Now we can approximate any interval $I$ of length $t\geq N^{-5/6}$ by an integral interval $I'$ with an error of at most two small intervals, so that $|T_I-T_{I'}|, |O_I-O_{I'}|\leq 200\log N$. Also, from $\E T_I=(x\coth x) N t$ and $\E O_I= Np_I=N\frac {\sinh (xt)\cosh(x(1-t))} {\sinh x}$ we see that $\E T_{I'}, \E O_{I'}\geq cN^{1 / 6}$ for a constant $c>0$ and $\frac {\E T_I}{\E T_{I'}},\frac {\E O_I}{\E O_{I'}} =1+o(1)$. Therefore, $T_I' \in [(1-3\epsilon)\E T_I', (1+3\epsilon)\E T_I']$ and $O_I' \in [(1-3\epsilon) \E O_I', (1+3\epsilon) \E O_I']$ will imply (respectively) $T_I \in [(1-4\epsilon) \E T_I, (1+4\epsilon)\E T_I]$ and $O_I \in [(1-4\epsilon)\E O_I, (1+4\epsilon)\E O_I]$, as desired.
	
	Now if $|I| > (1/2+6\epsilon)$, by the concentration of $T_I$ discussed above, we have
	$$T_I \geq (1 - 4\epsilon)\E T_I = (1 - 4\epsilon)(x \coth x) N |I|  > \alpha(1/2 + \epsilon)N$$
	for all  sufficiently small but fixed $\epsilon$. And if $|I|<N^{-5 / 6}$, then
	$$T_I \leq T_{I^*}\leq (1+4\epsilon) \E T_{I^*}<N^{\frac 1 5}$$
	where $I^*\supseteq I$ is an interval of length $N^{-5 / 6}$.
	Therefore, we have $N^{\frac 1 5}\leq T_I\leq\alpha(1/2+\epsilon)N$ implies $|I|\in [N^{-5/6},(1/2 + 6\epsilon)]$. However, if $|I|\in [N^{-5/6},(1/2 + 6\epsilon)]$, then by the concentration of $O_I$, we have $O_I\leq (1+4\epsilon) \E O_I= (1+4\epsilon)Np_I\leq(1/2+\epsilon_1)N$ for $\epsilon_1 = \epsilon^{1/2}$. Therefore, we see that
	\begin{equation}\label{eq-Case-3-1}
	O_I\leq (1/2+\epsilon_1)N, \mbox{ if } N^{\frac 1 5}\leq T_I\leq\alpha(1/2+\epsilon)N\,.
	\end{equation}
	A similar argument shows that for $\epsilon_2 = \epsilon^{1/4}$, $T_I>\alpha(1/2+\epsilon_2)N$ implies $|I|> (1/2+6\epsilon_1)$, which in turn implies $O_I> (1/2+\epsilon_1)N$. Therefore
	\begin{equation}\label{eq-Case-3-2}
	O_I> (1/2+\epsilon_1)N, \mbox{ if } T_I>\alpha(1/2+\epsilon_2)N\,.
	\end{equation}
	Finally, $N^{\frac 1 5}\leq T_I\leq\alpha(1/2+\epsilon_2)N$ implies $|I|\in [N^{-5/6},(1/2 + 6\epsilon_2)]$. But for $|I| \in [N^{-5/6},(1/2 + 6\epsilon_2)]$ we have
	$p_I={\frac {\sinh (x|I|)\cosh(x(1 - |I|))} {\sinh x}} \geq (x\coth x) |I| {\frac 1{\alpha+\epsilon_3'}}$ for $\epsilon_3' = 0.1\epsilon^{1/8}$,
	i.e.,
	$$\E O_I\geq \frac 1 {\alpha + \epsilon_3'}\E T_I\,.$$ By our assumptions on the concentration of $O_I$ and $T_I$ again, we deduce that $O_I \geq \frac 1{\alpha+\epsilon_3}T_I$ for $\epsilon_3 = \epsilon^{1/8}$. In other words
	\begin{equation}\label{eq-Case-3-3}
	O_I\geq {\frac {T_I}{\alpha+\epsilon_3}}, \mbox{ if } N^{\frac 1 5}\leq T_I \leq \alpha(1/2+\epsilon_2)N\,.
	\end{equation}
	By \eqref{eq-Case-3-1}, \eqref{eq-Case-3-2} and \eqref{eq-Case-3-3} we have completed the task of Case 3.
	
	\bigskip
	
	Combining the above three cases, we have completed the proof of \eqref{eq-good-1}, and thus the proof of the lemma.\qed
\end{proof}

Let $\mathcal{P}$ be the collection of good paths. For any path $P\in\mathcal{P}$, let $A_P$ be the event that $P$ is accessible. So we have $Z_{N,x,*}=\sum_{P\in\mathcal{P}} 1_{A_P}$. Notice that
\begin{eqnarray}\label{eq-second-moment}
\E Z_{N,x,*}^2 &=&\sum_{P\in \mathcal{P}}\sum_{P'\in \mathcal{P}}\P(A_{P}\cap A_{P'})\nonumber\\
&=&\sum_{P\in \mathcal{P}}\P(A_{P})\sum_{P'\in \mathcal{P}}\P(A_{P'} \mid A_{P})\nonumber\\
&=&\sum_{P\in \mathcal{P}}\P(A_{P})\E(Z_{N,x,*} \mid A_{P})\,.
\end{eqnarray}
So in order to estimate  $\E Z^2_{N,x,*}$, a key step is to estimate $\E(Z_{N,x,*} \mid A_{P})$. For any good path $P$ of length $L$, let $v_0=\vec{0}$, $v_1$, $v_2$, $\ldots$ , $v_L=\vec{1}$ be the $(L+1)$ vertices it passes through. Let $X_i$ be the (random) value at $v_i$ (recall that $X_0=0$ and $X_L=x$). We denote the successive differences of $X_i$'s by $\delta_1=X_1$, $\delta_2=X_2-X_1$, $\cdots$, $\delta_{L}=x-X_{L-1}$. It is clear that conditioning on $P$ to be accessible, the $X_i$'s are distributed as the order statistics of $(L-1)$ i.i.d.\ uniform $[0,x]$ random variables, so that the conditional distribution of $(\delta_1/x,\delta_2/x,\cdots, \delta_{L}/x)$ given $A_P$ is the Dirichlet distribution $\text {Dir}(1,1,\cdots,1)$. Recall that a Dirichlet distribution $\text {Dir}(\alpha_1,\alpha_2, \cdots, \alpha_K)$ is supported on $(x_1,x_2,\cdots,x_K)$ where $x_i\in[0,1]$ for all $i=1,\ldots, K$ and $\sum_{i=1}^K x_i=1$, and has a density $\frac {\Gamma(\sum_{i=1}^K \alpha_i)}{\prod_{i=1}^K \Gamma(\alpha_i)}\prod_{i=1}^K x_i^{\alpha_i-1}$.

We first state some properties of $(\delta_1,\delta_2, \cdots , \delta_{L})$ conditioning on $A_P$ (they are also known as \emph {the spacings} of the order statistics).
\begin{proposition}\label{prop-spacings-moment}
	For $0=i_0 < i_1 < i_2 < \cdots < i_k < i_{k+1}=L$ and nonnegative integers $\beta_1, \beta_2, \cdots, \beta_{k+1}$,
	\begin{enumerate}[(i)]
		\item Conditional on the event $A_P$, the distribution of 
		$$\frac{1}{x}(X_{i_1}-X_0, X_{i_2}-X_{i_1}, \cdots , X_L-X_{i_k}) = \frac{1}{x}\big(\sum\limits_{i=1}^{i_1}\delta_i, \sum\limits_{i=i_1+1}^{i_2}\delta_i, \cdots , \sum\limits_{i=i_k+1}^{L}\delta_i\big)$$
		is the Dirichlet distribution $\mathrm{Dir}(i_1,i_2-i_1,\cdots,L-i_k)$.
		\item $\E (\prod_{j=1}^{k+1} (X_{i_j}-X_{i_j-1})^{\beta_j} \mid A_P) \leq \prod_{j=1}^{k+1}  \E ((X_{i_j}-X_{i_j-1})^{\beta_j} \mid A_P)$.
		\label{prop-spacings-moment-2}
		\item $\E ((X_{i_1}-X_0)^{\beta_1} \mid A_P) \leq C\sqrt{1+t}(x{\frac {i_1-1}{L-1}}{\frac{(1+t)^{1+1/t}}e})^{\beta_1}$ for $ \beta_1\leq t (i_1 - 1)$, where $C>0$ is an absolute constant.\label{prop-spacings-moment-3}
	\end{enumerate}
\end{proposition}

\begin{proof}
	(i) This follows from the aggregation property of the Dirichlet distribution.\\
	(ii) This follows from the moments of Dirichlet-distributed random variables. That is, for $Y\sim \text {Dir}(\alpha_1,\alpha_2,\cdots,\alpha_K)$, we have
	\begin{equation*}\label{dirichlet-moment}
	\E (\prod_{j=1}^K Y_j^{\beta_j}) ={\frac{\Gamma(\sum_{j=1}^K \alpha_j)}{\Gamma(\sum_{j=1}^K \alpha_j+\beta_j)}}\prod_{j=1}^K{\frac {\Gamma(\alpha_j+\beta_j)} {\Gamma(\alpha_j)}}\leq \prod_{i=1}^K{\frac {\Gamma(\sum_{j=1}^K \alpha_j)}{\Gamma(\beta_i+\sum_{j=1}^K \alpha_j)}}\prod_{j=1}^K{\frac {\Gamma(\alpha_j+\beta_j)} {\Gamma(\alpha_j)}}= \prod_{j=1}^K\E(Y_j^{\beta_j})
	\end{equation*}
	where the inequality follows from the convexity of $\log \Gamma(x)$ for $x>0$ and induction.\\
	(iii) As a special case of the moments of Dirichlet-distributed random variables, we have
	\begin{eqnarray}\label{eq-beta-distribution}
	\E ((X_{i_1}-X_0)^{\beta_1} \mid A_P) =x^{\beta_1}{\frac {\Gamma(L)} {\Gamma(L+\beta_1)}}{\frac {\Gamma (i_1+\beta_1)} {\Gamma(i_1)}}
	=x^{\beta_1}{\frac {(L-1)!} {(L+\beta_1-1)!}}{\frac {(i_1+\beta_1-1)!} {(i_1-1)!}}\,.
	\end{eqnarray}
	By Stirling's formula, we have for an absolute constant $C>0$
	\begin{align*}
	\E ((X_{i_1}-X_0)^{\beta_1} \mid A_P) &\leq Cx^{\beta_1}{\frac {\sqrt{(L-1)}({\frac {L-1} e})^{L-1}} {\sqrt{(L+\beta_1-1)}({\frac {L+\beta_1-1} e})^{L+\beta_1-1}}}
	{\frac {\sqrt{(i_1+\beta_1-1)}({\frac {i_1+\beta_1-1} e})^{i_1+\beta_1-1}} {\sqrt{(i_1-1)}({\frac {i_1-1} e})^{i_1-1}}} \\
	& = C\big(x{\frac {i_1-1} {L-1}}\big)^{\beta_1}{\frac {\sqrt{(L-1)(i_1+\beta_1-1)}} {\sqrt{(L+\beta_1-1)(i_1-1)}}}
	\Bigg(\frac {(1+\frac {\beta_1}{i_1-1})^{1+\frac {i_1-1} {\beta_1}}} {(1+\frac {\beta_1}{L-1})^{1+\frac {L-1} {\beta_1}}}\Bigg)^{\beta_1}\,.
	\end{align*}
	Now by our assumption, we have ${\frac {(L-1)(i_1+\beta_1-1)} {(L+\beta_1-1)(i_1-1)}}\leq {\frac {i_1+\beta_1-1} {i_1-1}}\leq 1+t$. In addition, since the function $(1+z)^{1+1/z}$ is increasing in $z$ and tends to $e$ as $z\to0$, we have $\frac {(1+\frac {\beta_1}{i_1-1})^{1+\frac {i_1-1} {\beta_1}}} {(1+\frac {\beta_1}{L-1})^{1+\frac {L-1} {\beta_1}}}\leq{\frac {(1+t)^{1+1/t}}e}$. Substituting these bounds into the preceding display completes the proof.\qed
\end{proof}

In order to compute $\E(Z_{N,x,*}\mid A_{P})$, we first calculate $\E(Z_{N,x,*}(\vec {0}, v_{i_1}, v_{i_2}, \ldots , v_{i_k},\vec {1})\mid A_{P})$, where $\vec {0}$, $v_{i_1}$, $v_{i_2}$, $\ldots$ , $v_{i_k}$ , $\vec {1}$ ($0=i_0 < i_1 < i_2 < \cdots < i_k < i_{k+1}=L$) are vertices on path $P$ and $Z_{N,x,*}(\vec {0}, v_{i_1}, v_{i_2}, \ldots , v_{i_k},\vec {1})$ counts the number of good accessible paths $P'$ that intersect $P$ (vertex wise) at $\vec {0}$, $v_{i_1}$, $v_{i_2}$, $\ldots$ , $v_{i_k}$ , $\vec {1}$. For ease of notation we let $v_{i_0}=\vec {0}$ and $v_{i_{k+1}}=\vec {1}$. Naturally these $(k+2)$ common vertices divide both $P$ and $P'$ into $(k+1)$ segments. The lengths of these segments on $P$ are $i_1$, $(i_2-i_1)$, $\ldots$ , $(L-i_k)$. Suppose that $P'$ visits these $(k+2)$ common vertices at its $j_0=0$-th, $j_1$-th, $\ldots$ , $j_{k+1}$-th steps.
Then on $A_P$ we have
\begin{eqnarray*}
	\P(A_{P'} \mid X_0,X_1,\cdots,X_L) = {\frac {{X_{i_1}}^{j_1-1}}{(j_1-1)!}}{\frac {(X_{i_2}-X_{i_1})^{j_2-j_1-1}} {(j_2-j_1-1)!}}\cdots {\frac {(x-X_{i_k})^{j_{k+1}-j_k-1}} {(j_{k+1}-j_k-1)!}}\,.
\end{eqnarray*}
By Part \eqref{prop-spacings-moment-2} of Proposition \ref{prop-spacings-moment} we have
\begin{eqnarray*}
	\P(A_{P'} \mid A_{P}) 
	&=&\E ( \P(A_{P'} \mid X_0,X_1,\cdots,X_L) \mid A_P)\\
	&=& \E\bigg[{\frac {{Y_{i_1}}^{j_1-1}} {(j_1-1)!}}\cdot {\frac {(Y_{i_2}-Y_{i_1})^{j_2-j_1-1}} {(j_2-j_1-1)!}}\cdots {\frac {(x-Y_{i_k})^{j_{k+1}-j_k-1}} {(j_{k+1}-j_k-1)!}}\bigg]\\
	&\leq& \E{\frac {{Y_{i_1}}^{j_1-1}} {(j_1-1)!}}\E {\frac {(Y_{i_2}-Y_{i_1})^{j_2-j_1-1}} {(j_2-j_1-1)!}}\cdots \E{\frac {(x-Y_{i_k})^{j_{k+1}-j_k-1}} {(j_{k+1}-j_k-1)!}}
\end{eqnarray*}
where $Y_0=0, Y_1, \cdots, Y_{L-1}, Y_L=x$ are distributed as the order statistics of $(L-1)$ i.i.d.\ uniform $[0, x]$ random variables. Therefore, we have
\begin{eqnarray}\label{product-of-F}
&&\E(Z_{N,x,*}(\vec {0}, v_{i_1}, v_{i_2}, \ldots, v_{i_k},\vec {1})\mid A_{P})\nonumber\\
& = & \sum\limits_{\substack{P'\in \mathcal{P},\\P' \text { intersects } P \text { at } \vec {0}, v_{i_1}, v_{i_2}, \ldots, v_{i_k},\vec {1}}}\P(A_{P'}\mid A_{P})\nonumber\\
&\leq & \sum\limits_{\substack{P'\in \mathcal{P},\\P' \text { intersects } P \text { at } \vec {0}, v_{i_1}, v_{i_2}, \ldots, v_{i_k},\vec {1}}}\prod_{\ell = 1}^{k+1} \E{\frac {(Y_{i_\ell}-Y_{i_{\ell-1}})^{j_\ell- j_{\ell-1}-1}} {(j_\ell- j_{\ell-1}-1)!}}\nonumber\\
&\leq & \prod_{\ell = 1}^{k+1} F(v_{i_{\ell-1}}, v_{i_\ell})
\end{eqnarray}
where $F(v_{i_{\ell-1}}, v_{i_\ell})$ is defined as follows.
\begin{defn}
	For $u,v\in H_N$, we say a path $P^*$ connecting $u$ to $v$ is a good segment from $u$ to $v$, if there exists at least one good path whose subpath from $u$ to $v$ is $P^*$. For any good path $P=v_0, v_1, \ldots , v_L$ and $0\leq i< j\leq L$, let $F(v_i,v_j)=\E G(v_i,v_j, Y_i,Y_j)$ where $G(v_i,v_j,y_i,y_j)$ is the conditional expectation of the number of good accessible segments from $v_i$ to $v_j$, given that $X_i = y_i$ and $X_j = y_j$. 
\end{defn}
\noindent Now summing inequality \eqref{product-of-F} over $i_1 , i_2 , \ldots , i_k$ and $k$, we have
\begin{equation}\label{eq-Z-F}
\E(Z_{N,x,*} \mid A_{P})\leq \sum_{k, i_1 , i_2 , \ldots , i_k}\prod_{\ell=1}^{k+1} F(v_{i_{\ell-1}}, v_{i_\ell})\,.
\end{equation}
We can further split the sum on the right hand side into two parts, according to whether $\max \{i_1 , (i_2-i_1) , \ldots , (L-i_k)\}>L/2$ (i.e. whether the longest segment on $P$ is longer than $L/2$). That is,
\begin{eqnarray}
&&\sum_{k,i_1,i_2,\ldots, i_k}\prod_{\ell=1}^{k+1} F(v_{i_{\ell-1}}, v_{i_\ell}) \nonumber\\
& = &\sum\limits_{\substack{k,i_1,i_2,\ldots, i_k,\\ \max \{i_1,(i_2-i_1),\ldots, (L-i_k)\}>L/2}}\prod_{\ell=1}^{k+1} F(v_{i_{\ell-1}}, v_{i_\ell}) + \sum\limits_{\substack{k,i_1,i_2,\ldots, i_k,\\ \max \{i_1,(i_2-i_1),\ldots, (L-i_k)\}\leq L/2}}\prod_{\ell=1}^{k+1} F(v_{i_{\ell-1}}, v_{i_\ell})\nonumber\\
&\leq & \Big(\sum_{d=0}^{L/2} \sum_{d_1+d_2=d} F(v_{d_1},v_{L-d_2})\Big)\prod_{j=0}^{L-1}  (\sum_{i=1}^{\frac L 2}F(v_j,v_{j+i}) )+\prod_{j=0}^{L-1} (\sum_{i=1}^{\frac L 2}F(v_j,v_{j+i}))\,. \label{eq-F}
\end{eqnarray}

To justify the last inequality, we first point out that $F(v_j,v_{j+1})$ is always 1 because the Hamming distance between a pair of vertices on a good path is 1 if and only if these two vertices are neighboring each other on the path. Given any $k$ and $0<i_1<i_2<\cdots<i_k<L$, we define $u_j(k, i_1,i_2,\ldots,i_k)$ for $j=0,1,\ldots, L-1$ as:
\begin{equation*}
u_j(k, i_1,i_2,\ldots,i_k)=\begin{cases}
v_{i_{\ell+1}}\,, & \mbox { if } j=i_\ell \mbox{ for some } 1\leq \ell\leq k \mbox{ and } i_{\ell+1}-i_{\ell}>1\\
v_{j+1}\,, & \mbox{ otherwise }
\end{cases}
\end{equation*}
Thus for any $k$ and $0<i_1<i_2<\cdots<i_k<L$
$$\prod_{\ell=1}^{k+1} F(v_{i_{\ell-1}}, v_{i_\ell})=\prod_{j=0}^{L-1} F(v_j,u_j)\,.$$
Moreover, it is not hard to verify that $\vec u:=(u_0,u_1,\cdots,u_{L-1})$ is an injective function of $(k, i_1, i_2, \ldots, i_k)$, i.e., for any $(k,i_1,i_2,\ldots,i_k)\neq(k',i'_1,i'_2,\ldots, i'_{k'})$ such that $0<i_1<i_2<\cdots<i_k<L$ and $0<i'_1<i'_2<\cdots<i'_{k'}<L$, $u_j(k, i_1,i_2,\ldots,i_k)=u_j(k',i'_1,i'_2,\ldots, i'_{k'})$ cannot hold for all $j=0,1,\ldots, L-1$. Therefore
\begin{align*}
\sum\limits_{\substack{k,i_1,i_2,\ldots, i_k,\\ \max \{i_1,(i_2-i_1),\ldots, (L-i_k)\}\leq L/2}}\prod_{\ell=1}^{k+1} F(v_{i_{\ell-1}}, v_{i_\ell})&= \sum\limits_{\substack{k,i_1,i_2,\ldots, i_k,\\ \max \{i_1,(i_2-i_1),\ldots, (L-i_k)\}\leq L/2}}\prod_{j=0}^{L-1} F(v_j,u_j)\\
&\leq \prod_{j=0}^{L-1} (\sum_{i=1}^{\frac L 2}F(v_j,v_{j+i}))\,.
\end{align*}
The other part of the inequality can be obtained similarly.

The following two lemmas are useful for bounding $\E (Z_{N, x, *} \mid A_P)$.
\begin{lemma}
	\label{prop-longest-segment}
	For any sufficiently small but fixed number $\epsilon>0$, there exist $C_2>0$ and an integer $N'>0$ which both depend only on $\epsilon$, such that for all $|x - x_0| \leq \epsilon^2$, $N>N'$ and any good path $P$ we have $\sum_{d=0}^{L/2} \sum_{d_1+d_2=d}F(v_{d_1},v_{L-d_2})\leq C_2N(\sinh x)^{N-1}\cosh x$.
\end{lemma}

\begin{lemma}
	\label{prop-less-than-half-segment}
	For any sufficiently small but fixed number $\epsilon>0$, there exist $C_3>0$ and an integer $N'>0$ which both depend only on $\epsilon$, such that for all $|x - x_0| \leq \epsilon^2$, $N>N'$, any good path $P$ and any $j$ we have $\sum_{i=1}^{\frac L 2}F(v_j,v_{j+i})\leq 1+{\frac {C_{3}} N}$.
\end{lemma}

\begin{corollary}
	\label{cor-second-moment}
	For any sufficiently small but fixed number $\epsilon>0$, there exist $C_4>0$ and an integer $N'>0$ which both depend only on $\epsilon$,
	such that for all $|x - x_0| \leq \epsilon^2$ and $N>N'$
	$$\E Z_{N,x,*}^2\leq (C_4N\sinh^{N-1}x\cosh x+C_4)N\sinh^{N-1}x\cosh x \,.$$
\end{corollary}
\begin{proof}
	Substituting the bounds from Lemmas ~\ref{prop-longest-segment} and ~\ref{prop-less-than-half-segment} into \eqref{eq-F} and using \eqref{eq-Z-F}, we see that
	\begin{align*}
	\E(Z_{N,x,*} \mid A_{P}) &\leq \sum_{k,i_1,i_2,\ldots, i_k}\prod_{\ell=1}^{k+1} F(v_{i_{\ell-1}}, v_{i_\ell}) \nonumber\\
	&\leq (C_2N(\sinh x)^{N-1}\cosh x+1)(1+{\frac {C_{3}} N})^{(1+\epsilon)\alpha N} \nonumber\\
	&\leq (C_2N(\sinh x)^{N-1}\cosh x+1)e^{C_3(1+\epsilon)\alpha}\,.
	\end{align*}
	Substituting the above inequality into \eqref{eq-second-moment} and applying the inequality
	$$\sum_{P\in \mathcal{P}}\P(A_{P})=\E Z_{N,x,*}\leq \E Z_{N,x}\leq N(\sinh x)^{N-1}\cosh x$$ (here the last inequality follows from Corollary \ref{cor-first-moment}), we complete the proof of the corollary.\qed
\end{proof}
In order to prove Lemmas ~\ref{prop-longest-segment} and ~\ref{prop-less-than-half-segment}, we need the following lemma.
\begin{lemma}\label{fact-monotone-decreasing}
	Suppose that $N\geq 7$, $s\geq 1$.  Let $g(y, s) = (\sinh y)^s(\cosh y)^{N-s}$. Then $\frac{\partial g}{\partial y}(y, s)$ is decreasing in $s$ for all fixed $y>0$.
\end{lemma}

\begin{proof}
	By a direct calculation
	\begin{align*}
	\frac{\partial g}{\partial y}(y, s)&=(\sinh y)^s(\cosh y)^{N-s}(s\coth y+(N-s)\tanh y) \\
	&=(\sinh y)^{-1}(\cosh y)^{N-1}(\tanh y)^s(s+N(\sinh y)^2)\,.
	\end{align*} Therefore it suffices to show that $(\tanh y)^s(s+N(\sinh y)^2)$ is decreasing in $s$. Taking the partial derivative with respect to $s$ we get 
	$$\frac{\partial }{\partial s}[(\tanh y)^s(s+N(\sinh y)^2)]=(\tanh y)^s+(\ln \tanh y)(\tanh y)^s(s+N(\sinh y)^2)\,,$$ so we only need to show that $(\coth y)^{(s+N(\sinh y)^2)}\geq e$. If $\coth y\geq e$, then plainly we have $(\coth y)^{(s+N(\sinh y)^2)}\geq (\coth y)^s \geq \coth y \geq e$. On the other hand, if $\coth y < e$, then $y>\arccoth e:=y_0$. Since $(\coth y)^{(\sinh y)^2}$ is increasing in $y$, we have $(\coth y)^{(\sinh y)^2}\geq (\coth y_0)^{(\sinh y_0)^2}=e^{\frac 1 {e^2-1}}\approx 1.17$. Therefore we have $(\coth y)^{(s+N(\sinh y)^2)}\geq (\coth y)^{7(\sinh y)^2}\geq (\coth y_0)^{7(\sinh y_0)^2}>e$ in this case.\qed
\end{proof}

\paragraph{Proof of Lemma~\ref{prop-longest-segment}}
	For $d_1$ and $d_2$ such that $d_1+d_2=d$, it is clear that the Hamming distance $H(v_{d_1}, v_{L-d_2})$ between $v_{d_1}$ and $v_{L-d_2}$  is greater than or equal to $N-d$. Therefore, by \eqref{eq-upper-bound-first-moment} and Lemma~\ref{fact-monotone-decreasing}, we have
	\begin{eqnarray*}
		F(v_{d_1},v_{L-d_2})
		&=& \E G(v_{d_1},v_{L-d_2}, Y_{d_1}, Y_{L-d_2})\\
		&\leq&\E ((\sinh y)^{H(v_{d_1}, v_{L-d_2})}(\cosh y)^{N-H(v_{d_1}, v_{L-d_2})})'|_{y=Y_{L-d_2}-Y_{d_1}}\\
		&\leq& \E ((\sinh y)^{N-d}(\cosh y)^d)'|_{y=Y_{L-d_2}-Y_{d_1}}\\
		&=& \E ((\sinh y)^{N-d}(\cosh y)^d)'|_{y=x-Y_d}
	\end{eqnarray*}
	where the last equality is because the distribution of $Y_{L-d_2}-Y_{d_1}$ does not depend on  $(d_1, d_2)$ provided the value of $d=d_1+d_2$.
	Writing out the derivative in the last step, we have the following estimate
	\begin{eqnarray*}
		F(v_{d_1},v_{L-d_2})&\leq &\E ((\sinh y)^{N-d-1}(\cosh y)^{d-1}((N-d)(\cosh y)^2+d(\sinh y)^2))|_{y=x-Y_d}\\
		&\leq& \E ((\sinh y)^{N-d-1}(\cosh y)^{d-1}N (\cosh y)^2)|_{y=x-Y_d}\\
		&\leq& N (\cosh x)^2\E (\sinh(x-Y_d))^{N-d-1}(\cosh x)^{d-1}\,.
	\end{eqnarray*}
	Since $\sinh(x-y)\leq\sinh x-{\frac {\sinh x} x}y\text{ for }0\leq y\leq x$, we have further
	\begin{eqnarray}
	F(v_{d_1},v_{L-d_2})&\leq& N (\cosh x)^2\E (\sinh x-{\frac {\sinh x}  x}Y_d)^{N-d-1}(\cosh x)^{d-1} \nonumber\\
	&=& N (\cosh x)^2 (\sinh x)^{N-d-1}(\cosh x)^{d-1}\E(1-{\frac {Y_d} x})^{N-d-1}\,. \label{eq-F-bound}
	\end{eqnarray}
	It remains to bound $\E(1-{\frac {Y_d} x})^{N-d-1}$.  Since $1-{\frac {Y_d} x}$ is the $(L-d)$th order statistic of $(L-1)$ i.i.d.\ uniform $[0,1]$ random variables, it has a $\text{Beta}(L-d,d)$ distribution. Thus by the moments of Beta-distributed random variables (or applying \eqref{eq-beta-distribution}) we have
	\begin{equation}\label{eq-Y}
	\E(1-{\frac {Y_d} x})^{N-d-1}=\prod_{r=0}^{N-d-2}{\frac {L-d+r} {L+r}}
	\end{equation}
	which can be further bounded by
	\begin{align}\label{eq-Y-bound}
	\prod_{r=0}^{N-d-2}{\frac {L-d+r} {L+r}}&\leq (1-\frac d {L+N-d-2})^{N-d-1} \nonumber\\
	&\leq (e^{-{\frac {N-d-1} {L+N-d-2}}})^d\leq ({\frac {0.995} {\coth x}})^d
	\end{align}
	for $d\leq 0.32N$, $\epsilon$ sufficiently small and $N$ sufficiently large (recall that $L\in[\alpha(1-\epsilon)N,\alpha(1+\epsilon)N]$ for a good path). Here we used the inequality $e^{-\frac{1-0.32}{\alpha+1-0.32}}\leq \frac {0.994} {\coth x_0}$ (by brute force calculation).
	
	For $0.32N \leq d \leq \alpha(1/2+\epsilon)N$, set $t=d/N$ and $s = L/N$. Then by Stirling's formula
	\begin{eqnarray*}
		\prod_{r=0}^{N-d-2}{\frac {L-d+r} {L+r}} \leq  C_5\prod_{r=1}^{N-d}{\frac {L-d+r} {L+r}}
		&\leq&  C_6{\frac {(L+N-2d)^{L+N-2d}L^L} {(L-d)^{L-d}(L+N-d)^{L+N-d}}}\\
		&=& C_6\Big(\big({\frac {(1+s-2t)^{1+s-2t}s^s}{(s-t)^{s-t}(1+s-t)^{1+s-t}}}\big)^{\frac 1 t}\Big)^{d}.
	\end{eqnarray*}
	Another brute force calculation gives
	$$ \Big({\frac {(1+\alpha-2t)^{1+\alpha-2t}\alpha^\alpha}{(\alpha-t)^{\alpha-t}(1+\alpha-t)^{1+\alpha-t}}}\Big)^{\frac 1 t} \leq \frac {0.999} {\coth x_0}$$
	for $t \leq \alpha(1/2+\epsilon)$ and $\epsilon$ sufficiently small. Since the function $h(y,t)$ given by
	$$h(y,t) = \Big({\frac {(1 + y - 2t)^{1 + y - 2t}y^y}{(y - t)^{y - t}(1 + y - t)^{1 + y - t}}}\Big)^{\frac 1 t}$$
	is uniformly continuous with respect to $(y,t)$ on $[1.0, 1.5] \times [0.2, 0.8]$, we have for $\epsilon$ sufficiently small (so that $s$ is sufficiently close to $\alpha$) and $0.32\leq t \leq \alpha(1/2+\epsilon)$
	$$\Big({\frac {(1 + s - 2t)^{1 + s - 2t}s^s}{(s - t)^{s - t}(1 + s - t)^{1 + s - t}}}\Big)^{\frac 1 t} \leq \frac {0.9999} {\coth x_0}\,.$$
	In addition, for $\epsilon$ sufficiently small, the right hand side of the above inequality is at most $0.99999 / \coth x$. So we get $\prod_{r=0}^{N-d-2}{\frac {L-d+r} {L+r}}\leq C_6({\frac {0.99999} {\coth x}})^d$ in this case. Combined with \eqref{eq-F-bound}, \eqref{eq-Y} and \eqref{eq-Y-bound}, this completes the proof of the lemma.\qed

\paragraph{Proof of Lemma~\ref{prop-less-than-half-segment}}
	Recall that $P = v_0, v_1, \ldots, v_L$ is a good path of length $L$. For an arbitrary $j$, we will bound $F(v_j, v_{j+i})$ in a number of regimes depending on the value of $i$, as follows.
	
	\bigskip
	\noindent{\bf Case (a): $i=1$.} Since for any good path (or good segment), the Hamming distance between a pair of vertices on the path is 1 if and only if these two vertices are neighboring each other on the path, we have $F(v_j,v_{j+1})=1$.
	
	\bigskip
	\noindent{\bf Case (b): $i=2$.} The Hamming distance between $v_j$ and $v_{j+2}$ is precisely 2 (since $P$ is good), and thus the length of any good segment connecting $v_j$ to $v_{j+2}$ is either 2 or 4. There are at most 2 such segments of length 2, and the probability for each of them to be accessible given $X_j=y_j$ and $X_{j+2}=y_{j+2}$ is $(y_{j+2}-y_j)$. Similarly, there are at most $(N{4\choose 2}2!)$ such segments of length 4, and the probability for each of them to be accessible given $X_j=y_j$ and $X_{j+2}=y_{j+2}$ is $\frac{(y_{j+2}-y_j)^3}{3!}$. Therefore,
	\begin{align*}
	G(v_j,v_{j+2},y_j,y_{j+2})&\leq 2(y_{j+2}-y_j)+(N{4\choose 2}2!)\frac{(y_{j+2}-y_j)^3}{3!}\\
	&= 2(y_{j+2}-y_j)+2N(y_{j+2}-y_j)^3.
	\end{align*}
	Combined with \eqref{eq-beta-distribution}, this yields that
	$$F(v_j, v_{j+2}) \leq 20/N\mbox{ for sufficiently large } N\,.$$
	
	\bigskip
	\noindent{\bf Case (c): $i=3$.} The Hamming distance between $v_j$ and $v_{j+3}$ is precisely 3 (since $P$ is good), and thus the length of any good segment connecting $v_j$ to $v_{j+3}$ is either 3 or 5. Similar to the previous case, we have
	\begin{align*}
	G(v_j,v_{j+3},y_j,y_{j+3})&\leq 3(y_{j+3}-y_j)^2+(N{5\choose 2}3!)\frac{(y_{j+3}-y_j)^4}{4!}\\
	&=3(y_{j+3}-y_j)^2+(\frac 5 2 N)(y_{j+3}-y_j)^4.
	\end{align*}
	Combined with \eqref{eq-beta-distribution}, this yields that
	$$F(v_j, v_{j+3}) \leq 1000 \cdot N^{-2}\mbox{ for sufficiently large } N\,.$$
	
	\bigskip
	\noindent{\bf Case (d): $4\leq i\leq N^{\frac 1 5}$}. By the definition of good path and good segment again, we see that all the possible values of the pair $(H(v_j,v_{j+i}), L(v_j,v_{j+i}))$ (where $L(v_j,v_{j+i})$ is the length of a good segment connecting $v_j$ to $v_{j+i}$) are $(i,i)$, $(i,i+2)$, $(i-2,i-2)$ and $(i-2,i)$. Therefore $G(v_j,v_{j+i},y_j,y_{j+i})$ is at most  $$i(y_{j+i}-y_j)^{i-1}+{\frac {N{i+2\choose 2}i!}{(i+1)!}}(y_{j+i}-y_j)^{i+1}+(i-2)(y_{j+i}-y_j)^{i-3}+ {\frac {N{i\choose 2}(i-2)!}{(i-1)!}}(y_{j+i}-y_j)^{i-1}\,.$$
	Combined with \eqref{eq-beta-distribution}, this yields that
	$$F(v_j, v_{j+4}) \leq  10^4 \cdot N^{-1} \mbox{ for sufficiently large } N\,,$$
	$$F(v_j, v_{j+5}) \leq  10^4 \cdot N^{-1} \mbox{ for sufficiently large } N\,,$$
	$$F(v_j, v_{j+6}) \leq  10^4 \cdot N^{-1} \mbox{ for sufficiently large } N$$
	and
	$$F(v_j, v_{j+i}) \leq  10^4 \cdot (i(\frac i N)^4+ Ni (\frac i N)^6) \leq 10^4\cdot N^{-2} \mbox{ for sufficiently large } N$$
	when $7\leq i\leq N^{\frac 1 5}$.
	
	\bigskip
	\noindent{\bf Case (e): $N^{\frac 1 5}\leq i\leq L/2$.} Recall the definitions of $\epsilon_1, \epsilon_2, \epsilon_3$ in \eqref{eq-def-epsilons}. By the definition of good path, we have ${\frac i {\alpha+\epsilon_3}}\leq H(v_j,v_{j+i})\leq(1/2+\epsilon_1)N$. Therefore (by the definition of good path again) any good segment that connects $v_j$ to $v_{j+i}$ must have length $L(v_j,v_{j+i})\leq\alpha(1/2+\epsilon_2)N$, so that $L(v_j,v_{j+i})$ also satisfies $L(v_j,v_{j+i})\leq (\alpha+\epsilon_3) H(v_j,v_{j+i})\leq (\alpha+\epsilon_3)i$. By Part \eqref{prop-spacings-moment-3} of Proposition ~\ref{prop-spacings-moment}, we have
	$$\E(Y_{j+i}-Y_{j})^{\ell-1}\leq C\sqrt{1+\alpha+\epsilon_3}(x{\frac {i-1} {L-1}}{\frac {(1+(\alpha+\epsilon_3))^{1+1/(\alpha+\epsilon_3)}} e})^{\ell-1} \text{ for } \ell\leq (\alpha+\epsilon_3)(i-1)+1\,.$$
	Therefore by \eqref{lem-M-k-n2} and Lemma~\ref{fact-monotone-decreasing}, we have
	\begin{align*}
	F(v_j,v_{j+i})
	&= \sum_{\substack {P^* \mbox{ is a good segment of length } \ell\\ \mbox { connecting } v_j \mbox { to } v_{j+i}}} \frac {\E(Y_{j+i}-Y_{j})^{\ell-1}}{(\ell-1)!}\\
	&\leq C\sqrt{1+\alpha+\epsilon_3}((\sinh y)^{H(v_j,v_{j+i})}(\cosh y)^{N-H(v_j,v_{j+i})})'|_{y=x{\frac {i-1} {L-1}}{\frac {(1+(\alpha+\epsilon_3))^{1+1/(\alpha+\epsilon_3)}} e}}\\
	&\leq C\sqrt{1+\alpha+\epsilon_3}((\sinh y)^{\frac i {\alpha+\epsilon_3}}(\cosh y)^{N-{\frac i {\alpha+\epsilon_3}}})'|_{y=x{\frac {i-1} {L-1}}{\frac {(1+(\alpha+\epsilon_3))^{1+1/(\alpha+\epsilon_3)}} e}}\\
	&\leq C_7N^2(\sinh y)^{\frac i {\alpha+\epsilon_3}}(\cosh y)^{N-{\frac i {\alpha+\epsilon_3}}}|_{y=x{\frac i{L-1}}{\frac {(1+(\alpha+\epsilon_3))^{1+1/(\alpha+\epsilon_3)}} e}}\,.
	\end{align*}
	
	Set $a={\frac {N(\alpha+\epsilon_3)} {L-1} }$, $c=x{\frac {(1+(\alpha+\epsilon_3))^{1+1/(\alpha+\epsilon_3)}} e}$, and  $c_0=x_0{\frac {(1+\alpha)^{1+1/\alpha}} e}\approx 1.39$. Clearly $c$ will be sufficiently close to $c_0$ if $\epsilon$ is sufficiently small. Let $t={\frac i {L-1}}$ (so that ${\frac {N^{\frac 1 5}} L}\leq t\leq 1/2$) and $h(t):=(\sinh(ct))^{\frac t  {\alpha+\epsilon_3}}(\cosh(ct))^{\frac N {L-1}-\frac t{\alpha+\epsilon_3}}$. Then the preceding inequality can be rewritten as $F(v_j,v_{j+i})\leq C_7N^2(h(t))^{L-1}$. In order to estimate $F(v_j, v_{j+i})$, we analyze the behavior of the function $h(t)$ as follows. By straightforward computation, we have
	\begin{align*}
	(\alpha+\epsilon_3)\ln h(t)&=t\ln\sinh(ct)+(a-t)\ln\cosh(ct)\,,\\
	((\alpha+\epsilon_3)\ln h(t))'&=\ln\sinh(ct)-\ln\cosh(ct)+ct\coth(ct)+c(a-t)\tanh(ct)
	\end{align*}
	and
	\begin{align*}
	((\alpha+\epsilon_3)\ln h(t))'' &=c\coth(ct)-c\tanh(ct)+c\coth(ct)-c\tanh(ct)-{\frac {c^2t}{(\sinh(ct))^2}}+{\frac {c^2(a-t)}{(\cosh(ct))^2}} \\
	&\geq 2c(\coth(ct)-\tanh(ct))-c^2t({\frac 1 {(\sinh(ct))^2}}+{\frac 1 {(\cosh(ct))^2}})\\
	&={\frac c {(\sinh(ct))^2(\cosh(ct))^2}}(\sinh(2ct)-ct\cosh(2ct))>0
	\end{align*}
	for $t\leq 1/2$ (since $ct\leq c/2<0.8$).
	
	Therefore $(\alpha+\epsilon_3)\ln h(t)$, and consequently $h(t)$ is convex up to $t=1/2$. Thus we have $h(t)\leq \max(h({\frac {N^{\frac 1 5}} L}), h(1/2))$, and so $F(v_j,v_{j+i})\leq C_7N^2 \max ((h({\frac {N^{\frac 1 5}} L}))^{L-1},(h(1/2))^{L-1})$. However, since $(h(1/2))^{2(\alpha+\epsilon_3)}=\sinh(\frac c 2)(\cosh(\frac c 2))^{2(\alpha+\epsilon_3){\frac N {L-1}}-1}$ which is sufficiently close to $\sinh(\frac {c_0} 2)\cosh(\frac {c_0} 2)=\frac 1 2\sinh(c_0)<1$ if $\epsilon$ is sufficiently small and $N$ is sufficiently large, we have $h(1/2)\leq p$ where $p$ is a constant strictly less than 1. Thus, $(h(1/2))^{L-1}\leq p^{L-1}$. On the other hand, $(h({\frac {N^{\frac 1 5}} L}))^{L-1} \leq (N^{-{\frac 3 5}})^{N^{\frac 1 5}}(1+N^{-\frac 8 5})^N$ for sufficiently large $N$. Thus we have for $N$ sufficiently large,
	$$F(v_j, v_{j+i}) \leq C_7 N^2 \max(p^{L-1}, (N^{-{\frac 3 5}})^{N^{\frac 1 5}}(1+N^{-\frac 8 5})^N) \,.$$
	
	\noindent {\bf Conclusion.}
	Summing $F(v_j, v_{j+i})$ over $1\leq i\leq L/2$ and applying the bounds we obtained in Cases (a), (b), (c), (d) and (e), we see that $\sum_{i=1}^{\frac L 2}F(v_j,v_{j+i})\leq 1+{\frac {C_3} N}$ for some $C_3>0$, completing the proof of the lemma.\qed

\begin{proposition}\label{eq-augmenting}
	There exists $0\leq K<1$ such that, if $\liminf\limits_{N\to\infty}\P(Z_{N,x_c+\epsilon_N}>0)\geq C$ for some constant $C\geq 0$ whenever $N\epsilon_N\to \infty$, then whenever $N\epsilon_N\to \infty$ we have 
	$$\liminf\limits_{N\to\infty}\P(Z_{N,x_c+\epsilon_N}>0)\geq 1-(1-C)K\,.$$
\end{proposition}
\begin{proof}
	Our strategy basically follows that of \cite{HM14}. First we pick four vertices $a_1$, $a_2$, $b_1$, $b_2$ satisfying: $a_1$ and $a_2$ are neighbors of $\vec{0}$ and have a value in $[0,\epsilon_N/3]$, $b_1$ and $b_2$ are neighbors of $\vec{1}$ and have a value in $[x-\epsilon_N/3,x]$, and none of the four pairs $(a_i,b_j)$ are antipodal. Since $N\epsilon_N\to\infty$, this can be achieved with probability $1-o_N(1)$.
	
	Without loss of generality assume that the only coordinates of $a_1, a_2, b_1$ and $b_2$ that are different from $\vec{0}$ or $\vec{1}$ are $1,2,3$ and $4$, respectively. Let $\tilde H_1$ and $\tilde H_2$ be the $(N-2)$ dimensional sub-hypercubes of $\{0,1\}^N$ formed by $a_1,b_1$ and $a_2,b_2$,  respectively. That is, $\tilde H_1$ is the sub-hypercube with the first coordinate being 1 and the third coordinate being 0, and $\tilde H_2$ is the sub-hypercube with the second coordinate being 1 and the fourth coordinate being 0. Let $H_2'$ be $\tilde H_2\setminus \tilde H_1$. Denote by $p_{\tilde H_1}$ and $p_{H_2'}$ the probabilities that there is an accessible path in $\tilde H_1$ (from $a_1$ to $b_1$) and $H_2'$ (from $a_2$ to $b_2$) respectively. From the disjointness (and hence independence) of $\tilde H_1$ and $H_2'$ we have $\P(Z_{N,x_c+\epsilon_N}>0)\geq1-(1-p_{\tilde H_1})(1-p_{H_2'})-o_N(1)$. Clearly $p_{\tilde H_1}\geq \P(Z_{N-2,x_c+\epsilon_N/3}>0)\geq C-o_N(1)$. 
	
	It remains to show that $p_{H_2'}$ is bounded from below by a positive constant $1-K$. To this end, we note that if we only consider the good path in $\tilde H_2$ (from $a_2$ to $b_2$) which only updates Coordinate $1$ and  Coordinate $3$ once and Coordinate $3$ is updated before Coordinate $1$ (that is, in the associated sequence the numbers 1 and 3 occur precisely once each and 3 occurs ahead of 1), such path must be contained in $H_2'$. Clearly, the number of such accessible paths has second moment less than  $\E Z^2_{N-2,x_c+\epsilon_N/3,*}$ and first moment within an absolute multiplicative constant of $\E Z_{N-2,x_c+\epsilon_N/3,*}$ (indeed, the first moment is at least $C_1(N-2)\sinh^{N-3}(x)\cosh x\cdot (\frac x {\sinh x})^2\cdot \frac 1 2$ where $x=x_c+\epsilon_N/3$). Combined with Lemma~\ref{lem-first-moment} and Corollary~\ref{cor-second-moment}, this yields that $p_{H_2'} \geq 1-K-o_N(1)$ for some constant $K<1$. This completes the proof of the proposition.\qed
\end{proof}
\paragraph{Proof of \eqref{eq-mainthm-lower}: antipodal case} Applying Proposition~\ref{eq-augmenting} recursively (starting from $C=0$) completes the proof of \eqref{eq-mainthm-lower}.\qed

At the end of this section, we provide
\paragraph{Proof of \eqref{eq-mainthm-critical}: antipodal case}
	For the lower bound, it suffices to consider $x=x_c - \Delta/N$. By Remark \ref{remark-3}, we have in this case $N(\sinh x)^{N-1}\cosh x \geq m_1(\Delta)$ where $m_1(\Delta)>0$ depends only on $\Delta$. Applying the second moment method and using Lemma ~\ref{lem-first-moment} and Corollary ~\ref{cor-second-moment}, we obtain that (for sufficiently large $N$)
	$$\P(Z_{N,x}>0)\geq \P(Z_{N,x,*}>0)\geq {\frac {(\E Z_{N,x,*})^2} {\E Z_{N,x,*}^2}}\geq c_1(\Delta)\,,$$
	where $c_1(\Delta)>0$ depends only on $\Delta$.
	
	For the upper bound, it suffices to consider $x = x_c + \Delta/N$. Let $K>0$ be a large number depending on $\Delta$ that we specify later. The idea is to condition on the values of the neighbors of $\vec{0}$. Let $u_1, u_2, \ldots, u_N$ be these neighbors. For $1\leq i \leq N$ and ${\frac K N}\leq y_i\leq x$, we upper bound the conditional probability that $\vec{1}$ is accessible from $u_i$ given $X_{u_i}=y_i$ by the corresponding first moment, which by \eqref{eq-upper-bound-first-moment} can be further bounded by $((\sinh t)^{N-1}\cosh t)'|_{t=x-y_i}\leq 2N(\sinh(x-y_i))^{N-2}$. Therefore
	\begin{eqnarray*}
		\P(Z_{N,x}=0)&\geq& \int_{\frac K  N}^1\int_{\frac K N}^1\cdots \int_{\frac K N}^1[1-2N(\sinh(x-y_1))^{N-2}1_{y_1\leq x}-\cdots-2N(\sinh(x-y_N))^{N-2}1_{y_N\leq x}]\,dy_1\,dy_2\cdots\,dy_N\\
		&=&(1-{\frac K N})^N-(1-{\frac K N})^{N-1}\int_{\frac K N}^1 2N^2(\sinh(x-y_1))^{N-2}1_{y_1\leq x}\,dy_1\,,
	\end{eqnarray*}
	where
	\begin{align*}
	\int_{\frac K N}^1 2N^2(\sinh(x-y_1))^{N-2}1_{y_1\leq x}\,dy_1
	&=\int_{\frac K N}^x 2N^2(\sinh(x-y_1))^{N-2}\,dy_1\\
	&=\int_{0}^{x_0-{\frac {\sqrt2} 2}{\frac {\ln N} N}+{\frac {\Delta} N}-{\frac K N}}2N^2(\sinh y)^{N-2}\,dy\\
	&\to \sqrt 2 e^{\sqrt2(\Delta-K)}
	\end{align*}
	Here the last step follows from \cite[problem 213 (in Part Two Chapter 5 section 2)]{PS98} by setting $\varphi(x)=1, h(x)=\ln \sinh x, a=0, \xi=x_0, \alpha=-{\frac {\sqrt2} 2}, \beta=\Delta-K$. Therefore $\liminf\limits_{N\to\infty}\P(Z_{N,x}=0)\geq e^{-K}(1-\sqrt 2 e^{\sqrt2(\Delta-K)})$, and we are done by choosing $K$ to be a large number depending on $\Delta$.\qed

\section{Accessibility percolation: general case} \label{sec:general}
Since most of our proof in the antipodal case carries over to the general case, in the following proof for the general case we will emphasize the parts that require nontrivial modification.

Fix $0<\beta<1$ throughout this section. Recall from the statement of Theorem~\ref{thm-main} that $f(x)=(\sinh x)^\beta(\cosh x)^{1-\beta}$, that $x_0$ is the unique root of $f(x)=1$ and that $x_c = x_0 - {\frac 1 {f'(x_0)}}{\frac {\ln N} N}$. We have
$$f'(x)=(\beta \coth x+(1-\beta)\tanh x)(\sinh x)^\beta(\cosh x)^{1-\beta},$$
so that $f'(x_0)=\beta \coth x_0+(1-\beta)\tanh x_0$. In addition, it is straightforward to check that $0<f''(x_0)<\infty$. The proof of \eqref{eq-mainthm-upper} resembles that in the antipodal case.

\paragraph{Proof of \eqref{eq-mainthm-upper}: general case} In light of \eqref{lem-M-k-n2} we denote by 
	$$M_{N,\beta,x} := \big((\sinh x)^{\beta N}(\cosh x)^{(1-\beta)N}\big)'=((f(x))^N)'=N(f(x))^{N-1}f'(x)\,.$$
	We have $M_{N,\beta,x}\asymp N (f(x))^N$ for, say $|x- x_0| \leq 1/10$. Since $\P(Z_{N,x}>0)$ is monotone in $x$, we can assume without loss of generality that $\epsilon_N\leq N^{-2/3}$. With this assumption, we have for $x=x_c\pm\epsilon_N=x_0 - {\frac 1 {f'(x_0)}}{\frac {\ln N} N}\pm\epsilon_N$,
	$$(x-x_0)^2 = ({\frac 1 {f'(x_0)}}{\frac {\ln N} N}\pm\epsilon_N)^2 = o(1/N)$$
	and thus
	\begin{align*}
	f(x) &= f(x_0) + f'(x_0)(x-x_0) +o(1/N)\\
	&=1-{\frac {\ln N} N}\pm f'(x_0)\epsilon_N+o(1/N)\,.
	\end{align*}
	Therefore, $M_{N,\beta,x_c-\epsilon_N}\to 0$ and $M_{N,\beta,x_c+\epsilon_N}\to \infty$ as $N\to\infty$. Combined with \eqref{eq-upper-bound-first-moment}, it gives that $\E Z_{N, x_c - \epsilon_N} \to 0$ as $N\to\infty$, yielding \eqref{eq-mainthm-upper}.\qed

We next turn to prove \eqref{eq-mainthm-lower}. To this end, we first need to revise the definition of good path. Let
$$\gamma = \beta x_0\coth x_0+(1-\beta)x_0\tanh x_0 = x_0f'(x_0)$$
as in statement \eqref{eq-Case-1} (it will play the role of $\alpha$). Also in the general case, by a similar calculation as equation \eqref{eq-prob-odd}, we see that the definition of $g(t)$ in \eqref{def-g(t)} should be modified as
$$g(t): = \beta{\frac {\sinh(x_0t)\cosh(x_0(1-t))}{\sinh x_0}}+(1-\beta){\frac {\sinh(x_0(1-t))\sinh(x_0t)}{\cosh x_0}}$$
so that $g(t)N$ still means the ``expected Hamming distance traveled by a path in time $t$''.
In addition, for a pair of vertices $u$ and $v$, we let $H'(u,v)$ be their Hamming distance restricted to the first $\beta N$ coordinates (i.e., the number of the first $\beta N$ coordinates at which $u$ differs from $v$).

\begin{defn}[general case]\label{def-good-general}
	Let $\epsilon>0$ be a sufficiently small fixed number to be selected and set $\epsilon_4 = \epsilon^{1/8}$. We say a path (or the associated update sequence) $v_0=\vec{0}_N = (0,0,\cdots,0), v_1, \ldots, v_{L-1}, v_L=(\vec 1_{\beta N}, \vec 0_{N-\beta N}) = (1,\cdots,1,0,\cdots,0)$ is good if the following holds:
	\begin{enumerate}[(a)]
		\item The total number of updates of the first $\beta N$ coordinates lies within
		$$[\beta x_0\coth x_0(1-\epsilon)N,\beta x_0\coth x_0(1+\epsilon)N]$$
		and the total number of updates of the last $(1-\beta)N$ coordinates lies within $$[(1-\beta)x_0\tanh x_0(1-\epsilon)N, (1-\beta)x_0\tanh x_0(1+\epsilon)N]\,.$$
		\item $H(v_i,v_j)=|i-j|, \text{ if } |i-j|=1,2,3.$
		\item For $|i-j|>3$ we have\\
		$H(v_i,v_j)=|i-j| \text{ or } |i-j|-2,  \text{ if } 4\leq|i-j|\leq N^{\frac 1 5};\\
		H'(v_i,v_j)\leq (1/2+\epsilon_1)\beta N, \text{ if } |i-j|\leq\gamma(1/2+\epsilon)N;\\
		H'(v_i,v_j)> (1/2+\epsilon_1)\beta N, \text{ if } |i-j|>\gamma(1/2+\epsilon_2)N;\\
		H(v_i,v_j)\geq {\frac {2g(1/2) |i-j|} {\gamma+\epsilon_3}}, \text{ if } N^{\frac 1 5}\leq |i-j|\leq \gamma(1/2+\epsilon_2)N.$
		\item Let $D(v_0,v_i)$ be the number of updates of the first $\beta N$ coordinates among the first $i$ updates, and $D(v_{L-i},v_L)$ be the number of updates of the first $\beta N$ coordinates among the last $i$ updates. Then both $D(v_0,v_i)$ and $D(v_{L-i},v_L)$ are less than or equal to $\delta i$ for any $i\leq L/2$, where $\delta:= \frac {\beta \coth x_0} {\beta \coth x_0+(1-\beta)\tanh x_0}+\epsilon_4$.
	\end{enumerate}
\end{defn}
As in the antipodal case, it is clear that a good path is self-avoiding. In addition, we have $L\in[\gamma(1-\epsilon)N,\gamma(1+\epsilon)N]$ by Property (a).
\begin{lemma}\label{lem-first-moment-general}
	For any sufficiently small but fixed number $\epsilon>0$, there exist $C_1'>0$ and an integer $N'>0$ which both depend only on $\epsilon$, such that for all $|x - x_0| \leq \epsilon^2$ and $N>N'$ we have
	\begin{equation}
	\E Z_{N,x,*}\geq C_1'M_{N,\beta,x}=C_1'N(f(x))^{N-1}f'(x)\,.
	\end{equation}
\end{lemma}
\begin{proof}
	Recall the definition of $\tilde \mu_{N,\beta}$ introduced in the statement \eqref{eq-Case-1}: For $i\in \{1, \ldots, \beta N\}$, let $U_i$ be i.i.d.\ random variables distributed as $F_1$, and independently for $i\in \{\beta N + 1, \ldots, N\}$, let $U_i$ be i.i.d.\ random variables distributed as $F_2$. Given the values of $U_1, \ldots, U_N$, we let $(A_1, \ldots, A_L)$ (where $L = \sum_{i=1}^N U_i$) be a sequence uniformly at random subject to $|\{1\leq j \leq L: A_j = i\}| = U_i$. Let $\tilde \mu_{N,\beta}$ be the probability measure of the random sequence $(A_1, \ldots, A_L)$ thus obtained.
	
	Following a similar argument given at the beginning of the proof of Lemma~\ref{lem-first-moment}, we see that it suffices to show that under $\tilde \mu_{N,\beta}$ the set of good sequences has probability bounded from below by a constant.
	
	We first observe that Properties (a) and (c) in Definition \ref{def-good-general} can be satisfied by a random sequence under $\tilde \mu_{N,\beta}$ with probability tending to 1 as $N\to \infty$. This can be derived quite similarly as Case 2 and Case 3 in the proof of Lemma~\ref{lem-first-moment}, with the last requirement in Property (c) hinted by the following inequality
	$${\frac {g(t)} t}\geq {\frac {g(1/2)} {1/2}}, \mbox{ if } 0 \leq t\leq{\frac 1 2}\,.$$
	
	In addition, we claim that  Properties (b) and (d) in Definition \ref{def-good-general} can be satisfied simultaneously by a random sequence under $\tilde \mu_{N,\beta}$ with probability bounded from below. Altogether, this would imply the desired bound in the lemma.
	
	To verify this claim, we show that the update sequence $(A_1, \ldots, A_L)$ can be obtained by the following two-step procedure, where in each step one property can be satisfied with probability bounded from below. Let us recall the notation that $\mathcal{F} = \sigma(U_1, U_2, \ldots, U_N)$. For convenience, we write  $L_1=\sum_{i=1}^{\beta N}U_i$ and $L_2=\sum_{i=\beta N+1}^{N}U_i$.

	As the first step, conditioning on $\mathcal{F}$, we choose $L_1$ indices $i_1<i_2<\cdots<i_{L_1}$ uniformly from $\{1,2,\ldots,L\}$ and call them type 1 (they represent updates of the first $\beta N$ coordinates). Denote by $\mathcal I=\{i_1,i_2,\cdots,i_{L_1}\}$ the collection of these type 1 indices. Let $j_1<j_2<\cdots<j_{L_2}$ be the rest of the indices and call them type 2 (they represent updates of the last $(1-\beta) N$ coordinates). In the following $\P$ refers to this (conditional) probability space (so that $L_1$ and $L_2$ should be seen as constants).
	
	Denote by $\mathcal E$ the following event:
	$$|\{1,\cdots,i\} \cap \mathcal I |, |\{L-i+1,\cdots, L \} \cap \mathcal I |  \leq \big(\frac {\beta \coth x_0} {\beta \coth x_0+(1-\beta)\tanh x_0}+\epsilon_4\big)i \mbox{ for all } 1\leq i \leq L/2$$
	and by $\mathcal E'$ the following event:
	$$|\{1,\cdots,i\} \cap \mathcal I |, |\{L-i+1,\cdots, L \} \cap \mathcal I |  \leq (\frac {L_1}{L_1+L_2}+\epsilon) i \mbox{ for all } 1\leq i \leq L/2\,.$$
	We want to show that Property (d) can be satisfied with probability bounded from below in this step, that is $\P(\mathcal E)\geq c$ for a constant $c>0$. Without loss we can assume that Property (a) holds (since it is $\mathcal F$-measurable and can be satisfied with high probability), so that we have 
	$$\tfrac {L_1}{L_1+L_2}+\epsilon\leq \tfrac {\beta x_0\coth x_0(1+\epsilon)N}{\beta x_0\coth x_0(1-\epsilon)N+(1-\beta)x_0\tanh x_0(1-\epsilon)N}+\epsilon\leq \tfrac {\beta \coth x_0} {\beta \coth x_0+(1-\beta)\tanh x_0}+\epsilon_4$$ for sufficiently small $\epsilon$ and therefore $\mathcal E'  \subseteq \mathcal E$. It thus remains to lower bound $\P(\mathcal E')$.
	
	To this end, for each $1\leq i\leq L$, we let $T_i = 1_{\{i \text{ is of type 1}\}}$. Then $T_1, T_2, \ldots, T_L$ can be viewed as a sample without replacement from $L_1$ 1's and $L_2$ 0's. By Hoeffding's inequality in the case of sampling without replacement \cite[Theorem 4]{H63}, we have for any $n$,
	$$\P(\frac {\sum_{i=1}^n T_i} {n} \geq \frac {L_1}{L_1+L_2}+\epsilon)\leq \exp (-2n\epsilon^2)$$
	and
	$$\P(\frac {\sum_{i=L-n+1}^L T_i} {n} \geq \frac {L_1}{L_1+L_2}+\epsilon)\leq \exp (-2n\epsilon^2)\,.$$
	By a union bound over $M \leq n \leq \frac L 2$ (where $M$ depending only on $\epsilon$ is chosen later), we have
	$\P(\mathcal E_1)\geq 1-\frac {2\exp (-2\epsilon^2 M)}{1-\exp (-2\epsilon^2)}$, where
	$$\mathcal E_1 = \Big\{\frac {\sum_{i=1}^n T_i} {n} \leq \frac {L_1}{L_1+L_2}+\epsilon \mbox{ and } \frac {\sum_{i=L-n+1}^L T_i} {n} \leq \frac {L_1}{L_1+L_2}+\epsilon \mbox{ for all }M \leq n \leq \frac L 2\Big\}\,.$$
	Let $\mathcal K$ be the set of all positive integer pairs $(k_1, k_2)$  such that $\frac {M-k_1} M\leq \frac {L_1}{L_1+L_2}+\epsilon$ and $\frac {M-k_2} M\leq \frac {L_1}{L_1+L_2}+\epsilon$. It is clear that
	\begin{equation}\label{mathcal-K-1}
	\mathcal E_1 = \bigsqcup_{(k_1, k_2)\in \mathcal K} \mathcal E_1 \cap \{\sum_{i=1}^M T_i=M-k_1\} \cap \{\sum_{i=L-M+1}^L T_i=M-k_2\}\,.
	\end{equation}
	For $(k_1, k_2)\in \mathcal K$, define
	\begin{align*}
	\mathcal E_2(k_1)& =\{T_i=0 \mbox{ for }1\leq i\leq k_1\} \cap \{ T_i=1 \mbox{ for } k_1+1\leq i\leq M \}\,,\\
	\mathcal E_3(k_2) & =\{T_i=0 \mbox{ for } L-k_2+1\leq i \leq L\} \cap \{ T_i=1 \mbox{ for } L-M+1\leq i\leq L-k_2 \}\,.
	\end{align*}
	Then for all $(k_1, k_2)\in \mathcal K$, on the event $\mathcal E_2(k_1) \cap \mathcal E_3(k_2)$ we have $\frac {\sum_{i=1}^n T_i} {n} \leq \frac {L_1}{L_1+L_2}+\epsilon$ and $\frac {\sum_{i=L-n+1}^L T_i} {n} \leq \frac {L_1}{L_1+L_2}+\epsilon$ for all $1\leq n\leq M$. Therefore, we have
	\begin{equation}\label{mathcal-K-2}
	\mathcal E'  \supseteq \bigsqcup_{(k_1, k_2) \in \mathcal K} \mathcal E_1 \cap \mathcal E_2(k_1) \cap \mathcal E_3(k_2)\,.
	\end{equation}
	However,
	\begin{align}
	\P(\mathcal E_1 \cap \mathcal E_2(k_1) \cap \mathcal E_3(k_2))&={M\choose{k_1}}^{-1}{M\choose{k_2}}^{-1}\P(\mathcal E_1, \sum_{i=1}^M T_i=M-k_1, \sum_{i=L-M+1}^L T_i=M-k_2) \nonumber\\
	&\geq 2^{-2M}\P(\mathcal E_1, \sum_{i=1}^M T_i=M-k_1, \sum_{i=L-M+1}^L T_i=M-k_2)\,. \label{eq-E-1-2-3}
	\end{align}
	Summing \eqref{eq-E-1-2-3} over all $(k_1, k_2)\in \mathcal K$ and using \eqref{mathcal-K-1} and \eqref{mathcal-K-2}, we deduce that 
	\begin{equation}
	\P(\mathcal E')\geq 2^{-2M} \P(\mathcal E_1)\geq 2^{-2M}(1-\frac {2\exp (-2\epsilon^2 M)}{1-\exp (-2\epsilon^2)})\,.\label{eq-prob-E}
	\end{equation}
	By \eqref{eq-prob-E} and choosing $M$ depending on $\epsilon$ (e.g. $M=-\frac {10} {\epsilon^2}\ln \epsilon$), we have proved that in the first step, Property (d) can be satisfied with probability bounded from below by a number depending only on $\epsilon$.
	
	Now, conditioning on the previous step, let $(B_1, B_2, \ldots, B_{L_1})$ be a sequence uniformly at random subject to $|\{1\leq j \leq L_1: B_j = i\}| = U_i$ for $i=1,2,\ldots,\beta N$, and \emph{independently} let $(C_1, C_2, \ldots, C_{L_2})$ be a sequence uniformly at random subject to $|\{1\leq j \leq L_2: C_j = i\}| = U_i$ for $i=\beta N+1,\beta N+2,\ldots, N$. Let $A_{i_k}=B_k$ for $1\leq k\leq L_1$ and $A_{j_k}=C_k$ for $1\leq k\leq L_2$ (recall that $i_k$'s and $j_k$'s are sampled in the previous step). Thanks to the general proof of Case 1 in Lemma~\ref{lem-first-moment}, we have that with high probability (with respect to the $U_i$'s), we have $B_i\neq B_{i+1}$ and $B_i\neq B_{i+2}$ hold for all $1\leq i\leq L_1$ with at least constant probability; and with high probability (with respect to the $U_i$'s), we have $C_i\neq C_{i+1}$ and $C_i\neq C_{i+2}$ hold for all $1\leq i\leq L_2$ with at least constant probability. However, note that $B_i\neq B_{i+1}$, $B_i\neq B_{i+2}$ for all $1\leq i \leq L_1$ and $C_i\neq C_{i+1}$, $C_i\neq C_{i+2}$ for all  $1\leq i\leq L_2$ together would imply $A_i\neq A_{i+1}$ and $A_i\neq A_{i+2}$ for all $1\leq i\leq L$, which corresponds to Property (b). By the (conditional) independence of $(B_1, B_2, \ldots, B_{L_1})$ and $(C_1, C_2, \ldots, C_{L_2})$, we see that in the second step, Property (b) can be satisfied with probability bounded from below by a constant.
	
	Finally, it is clear that the sequence $(A_1, \ldots, A_L)$ obtained by this two-step procedure has the same distribution as under $\tilde \mu_{N,\beta}$ originally. This completes the verification of our claim and therefore the lemma.\qed
\end{proof}

\begin{lemma}\label{lem-longest-segment-general}
	For any sufficiently small but fixed number $\epsilon>0$, there exist $C_2'>0$ and an integer $N'>0$ which both depend only on $\epsilon$, such that for all $|x - x_0| \leq \epsilon^2$, $N>N'$ and any good path $P = v_0, v_1, \ldots, v_L$ we have
	$$\sum_{d=0}^{L/2} \sum_{d_1+d_2=d}F(v_{d_1},v_{L-d_2})\leq C_2'Nf(x)^N \asymp C_2'N(f(x))^{N-1}f'(x)\,.$$
\end{lemma}

\begin{proof}
	We continue to let $Y_0=0, Y_1, \ldots, Y_{L-1}, Y_L=x$ be distributed as the order statistics of $(L-1)$ i.i.d.\ uniform $[0, x]$ random variables. For $d_1$ and $d_2$ such that $d_1+d_2=d$, by Property (d) of Definition~\ref{def-good-general} we have that the Hamming distance $H(v_{d_1}, v_{L-d_2})$ between $v_{d_1}$ and $v_{L-d_2}$ is at least $\beta N-D(v_0,v_{d_1})-D(v_{L-d_2},v_L)$, which is at least $\beta N-\delta d$. Therefore, by \eqref{eq-upper-bound-first-moment} and Lemma~\ref{fact-monotone-decreasing}, we have
	\begin{align}\label{eq-1}
	F(v_{d_1},v_{L-d_2})
	&= \E G(v_{d_1},v_{L-d_2}, Y_{d_1}, Y_{L-d_2})\nonumber\\
	&\leq\E ((\sinh y)^{H(v_{d_1}, v_{L-d_2})}(\cosh y)^{N-H(v_{d_1}, v_{L-d_2})})'|_{y=Y_{L-d_2}-Y_{d_1}} \nonumber \\
	&\leq\E ((\sinh y)^{\beta N-\delta d}(\cosh y)^{(1-\beta)N+\delta d})'|_{y=Y_{L-d_2}-Y_{d_1}} \nonumber \\
	&=\E ((\sinh y)^{\beta N-\delta d}(\cosh y)^{(1-\beta)N+\delta d})'|_{y=x-Y_d}
	\end{align}
	where the last equality is because the distribution of $Y_{L-d_2}-Y_{d_1}$ does not depend on  $(d_1, d_2)$ provided the value of $d=d_1+d_2$. Since $x-Y_d$ is the $(L-d)$th order statistic of $(L-1)$ i.i.d.\ uniform $[0,x]$ random variables, $\frac {x- Y_d} x $ has a $\text{Beta}(L-d,d)$ distribution. Thus, the density of $x-Y_d$ is ${\frac 1 x}({\frac y x})^{L-d-1}(1-{\frac y x})^{d-1}{\frac {(L-1)!}{(L-d-1)!(d-1)!}}$ for $y\in[0,x]$. Therefore
	\begin{align}\label{eq-2}
	&\E ((\sinh y)^{\beta N-\delta d}(\cosh y)^{(1-\beta)N+\delta d})'|_{y=x-Y_d} \nonumber\\
	=&\int_0^x ((\sinh y)^{\beta N-\delta d}(\cosh y)^{(1-\beta)N+\delta d})'{\frac 1 x}({\frac y x})^{L-d-1}(1-{\frac y x})^{d-1}{\frac {(L-1)!}{(L-d-1)!(d-1)!}}\,dy\,.
	\end{align}
	
	We will split the above integral into two parts according to whether $y$ is smaller or greater than $\frac x 2$, and denote by $\mathcal J_1(d)$ the integral over $[0, \frac x2]$ and by $\mathcal J_2(d)$ the integral over $[\frac x2, x]$.
	On one hand, for $y\in[0, {\frac x 2}]$, by Lemma~\ref{fact-monotone-decreasing} we have
	\begin{eqnarray*}
		&&((\sinh y)^{\beta N-\delta d}(\cosh y)^{(1-\beta)N+\delta d})' \leq ((\sinh y)^{\beta N-\delta\frac L 2}(\cosh y)^{(1-\beta)N+\delta \frac L 2})' \\
		&=& (\sinh y)^{\beta N-\delta\frac L 2-1}(\cosh y)^{(1-\beta)N+\delta \frac L 2-1} ((\beta N-\delta\frac L 2)\cosh y+ ((1-\beta)N+\delta \frac L 2)\sinh y)\,.
	\end{eqnarray*}
	Since $\cosh y\leq \cosh (\frac x 2)$ and $\sinh y\leq \sinh (\frac x 2)$ for $y\in[0, {\frac x 2}]$, we have
	\begin{eqnarray*}
		((\sinh y)^{\beta N-\delta d}(\cosh y)^{(1-\beta)N+\delta d})' 
		&\leq& C_8N(\sinh (\frac x 2))^{\beta N-\delta\frac L 2-1}(\cosh (\frac x 2))^{(1-\beta)N+\delta \frac L 2-1}\\
		&\leq& C_8N(\sinh (\frac x 2))^{\beta N-\delta{\frac {\gamma(1+2\epsilon)N} 2}}(\cosh (\frac x 2))^{(1-\beta)N+\delta{\frac {\gamma(1+\epsilon)N} 2}}\,,
	\end{eqnarray*}
	where the last inequality follows from Property (a) of Definition~\ref{def-good-general}.
	Therefore
	\begin{eqnarray}\label{eq-5}
	\sum_{d=1}^{\frac L 2}(d+1)\mathcal J_1(d) &\leq& C_9N^3(\sinh (\frac x 2))^{\beta N-\delta{\frac {\gamma(1+2\epsilon)N} 2}}(\cosh (\frac x 2))^{(1-\beta)N+\delta{\frac {\gamma(1+\epsilon)N} 2}} \nonumber\\
	&\leq& C_9 N^3 r^N (\sinh x)^{\beta N}(\cosh x)^{(1-\beta)N}\,,
	\end{eqnarray}
	where $0<r<1$ is a constant that depends only on $\beta$. Here we used the fact (by brute force computation) that
$$\frac{(\sinh(\frac x 2))^{\beta-\delta{\frac {\gamma(1+2\epsilon)} 2}}(\cosh(\frac x 2))^{(1-\beta)+\delta{\frac {\gamma(1+\epsilon)} 2}}}{(\sinh x)^\beta(\cosh x)^{1-\beta}}\leq r<1\,.$$
	
	On the other hand, for $y\in[{\frac x 2},x]$, we have $\coth y\leq \coth ({\frac x 2})$ and $\tanh y\leq \tanh (x)$. Thus
	\begin{eqnarray*}
		&&((\sinh y)^{\beta N-\delta d}(\cosh y)^{(1-\beta)N+\delta d})'{\frac 1 x}\\
		&=&(\sinh y)^{\beta N-\delta d}(\cosh y)^{(1-\beta)N+\delta d}((\beta N-\delta d)\coth y+ ((1-\beta)N+\delta d)\tanh y){\tfrac 1 x}\\
		&\leq& C_{10}N (\sinh y)^{\beta N-\delta d}(\cosh y)^{(1-\beta)N+\delta d}\\
		&\leq& C_{11}N((\sinh y)^{\beta N}(\cosh y)^{(1-\beta)N})(\coth y)^{\delta (d-2)}\,.
	\end{eqnarray*}
	Therefore, the integrand of \eqref{eq-2} is smaller than $C_{11}N((\sinh y)^{\beta N}(\cosh y)^{(1-\beta)N}) \varphi(x, y, d, \beta, N, L)$ for $y\in [\frac x2, x]$, where
	\begin{eqnarray*}
\varphi(x, y, d, \beta, N, L)= (\coth y)^{\delta (d-2)}({\frac y x})^{L-d-1}(1-{\frac y x})^{d-1}{\frac {(L-1)!}{(L-d-1)!(d-1)!}}\,,
	\end{eqnarray*}
	and thus
	\begin{eqnarray}\label{eq-3}
	\sum_{d=1}^{\frac L 2}(d+1)  \mathcal J_2(d) \leq C_{11}N\int_{\frac x 2}^x ((\sinh y)^{\beta N}(\cosh y)^{(1-\beta)N})\sum_{d=1}^{\frac L 2}  (d+1) \varphi(x, y, d, \beta, N, L)\,dy\,.
	\end{eqnarray}
	Now for $d=1$, we have
	\begin{eqnarray}\label{d is 1}
	(d+1) \varphi(x, y, d, \beta, N, L) = 2(\tanh y)^\delta ({\tfrac y x})^{L-2}(L-1) 
	\leq C_{12} ({\tfrac y x})^{L-2}(L-1)\,.
	\end{eqnarray}
	In addition, for $d\geq 2$, we have
	\begin{eqnarray*}
		&&(d+1) \varphi(x, y, d, \beta, N, L) \\
		&=& (1-{\frac y x})(\coth y)^{\delta (d-2)}(1-{\frac y x})^{d-2}({\frac y x})^{L-d-1} {\frac {(L-3)!} {(L-d-1)!(d-2)!}} (L-1)(L-2)\frac {(d+1)}{(d-1)} \\
		&\leq& 3(1-{\frac y x}) L^2 \cdot  [((\coth y)^{\delta} )(1-{\frac y x})]^{d-2}({\frac y x})^{L-d-1} {\frac {(L-3)!} {(L-d-1)!(d-2)!}} \,.
	\end{eqnarray*}
	Observing that the second factor of the product in the previous line is a binomial term, we have
	\begin{eqnarray} \label{d geq 2}
	\sum_{d=2}^{\frac L 2}(d+1) \varphi(x, y, d, \beta, N, L) \leq 3 (1-{\tfrac y x})L^2\cdot ((\coth y)^\delta(1-{\tfrac y x})+{\tfrac y x})^{L-3}\,.
	\end{eqnarray}
	Combining \eqref{d is 1} and \eqref{d geq 2} and using Property (a) of Definition~\ref{def-good-general}, we have
	$$\sum_{d=1}^{\frac L 2}(d+1) \varphi(x, y, d, \beta, N, L) \leq C_{12}({\frac y x})^{\gamma(1-2\epsilon)N}(L-1)+3L^2(1-{\frac y x})((\coth y)^\delta(1-{\frac y x})+{\frac y x})^{\gamma(1+\epsilon)N}\,.$$
	Therefore \eqref{eq-3} translates to
	\begin{eqnarray}\label{eq-4}
	\sum_{d=1}^{\frac L 2}(d+1)\mathcal J_2(d)&\leq & C_{11}N \int_{\tfrac x 2}^x ((\sinh y)^{\beta N}(\cosh y)^{(1-\beta)N})\cdot\nonumber\\
	&&(C_{12}({\tfrac y x})^{\gamma(1-2\epsilon)N}(L-1)+3L^2(1-{\tfrac y x})((\coth y)^\delta(1-{\tfrac y x})+{\tfrac y x})^{\gamma(1+\epsilon)N})\,dy\,.\nonumber
	\end{eqnarray}
	For convenience, let
	\begin{align*}
	\psi_1(y) &= \frac {((\sinh y)^{\beta}(\cosh y)^{1-\beta})\cdot({\frac y x})^{\gamma(1-2\epsilon)}} {(\sinh x)^{\beta}(\cosh x)^{1-\beta}}\,,\\
	\psi_2(y) &=\frac {((\sinh y)^{\beta}(\cosh y)^{1-\beta})\cdot((\coth y)^\delta(1-{\frac y x})+{\frac y x})^{\gamma(1+\epsilon)}}{(\sinh x)^{\beta}(\cosh x)^{1-\beta}}\,,\\
	\psi_3(y) &=\beta \ln \sinh y+(1-\beta) \ln \cosh y\,.
	\end{align*}
	We will show that for $i\in\{1,2\}$, we have $\psi_i(x)=1$ (which is trivial) and $\ln \psi_i(y)\leq -K(x-y)$ for $y\in [\frac x 2, x]$, where $K>0$ is a constant that only depends on $\beta$. These conditions on $\psi_1(y)$ and $\psi_2(y)$ will guarantee that both $\int_{\frac x 2}^x N(\psi_1(y))^N\,dy$ and $\int_{\frac x 2}^x N^2(x-y)(\psi_2(y))^N\,dy$ are bounded as $N\to\infty$, so that \eqref{eq-4} is bounded by $C_{13}N(\sinh x)^{\beta N}(\cosh x)^{(1-\beta)N}$.
	
	It is relatively easy to check that $\psi_1(y)$ satisfies the second condition (i.e. $\ln \psi_1(y)\leq -K(x-y)$), so we focus on verifying it for $\psi_2(y)$. To start with, we have
	$$\ln \psi_2(y)=(\psi_3(y)-\psi_3(x))+\gamma(1+\epsilon)\ln ((\coth y)^\delta(1-{\frac y x})+{\frac y x})\,.$$
	For the first part of the sum on the right hand side, i.e. $(\psi_3(y)-\psi_3(x))$, we can first compute the derivatives of $\psi_3(y)$ as follows:
	\begin{align*}
	\psi_3'(y)&=\beta \coth y+(1-\beta) \tanh y\,,\\
	\psi_3''(y)&=\beta (1-(\coth y)^2) + (1-\beta) (1-(\tanh y)^2)\,.
	\end{align*}
	Since $\coth y\geq \coth x\geq (\frac{1-\beta} \beta)^{\frac 1 4}$ for $y\leq x$, we have $\psi_3''(y)$ is increasing in $y$. Therefore, by Taylor's theorem (Lagrange form of the remainder) we find that for some  $ \xi \in [y, x]$, 
	\begin{align}\label{psi_2(y)-part1}
	\psi_3(y) &=\psi_3(x)+\psi_3'(x)(y-x)+\frac {\psi_3''(\xi)}2 (y-x)^2\nonumber\\
	&\leq \psi_3(x)+\psi_3'(x)(y-x)+\frac {\psi_3''(x)}2 (y-x)^2\,.
	\end{align}
	For the second part of the sum, i.e. $\gamma(1+\epsilon)\ln ((\coth y)^\delta(1-{\frac y x})+{\frac y x})$, we set $C(y):=(\coth y)^\delta-1$ and $\theta(y):=1-\frac y x$. Then 
	$$(\coth y)^\delta(1-{\frac y x})+{\frac y x}=1+(1-{\frac y x})C(y)=1+\theta(y)C(y)\,.$$ Clearly $0\leq \theta(y)C(y)\leq \theta(\frac x 2)C(\frac x 2)=\frac 1 2 ((\coth \frac x 2)^\delta-1)\leq 0.75$. Since $\ln (1+t)\leq t-\frac {t^2} 3$ for $0\leq t\leq 0.75$, we have
	\begin{equation}\label{psi_2(y)-part2}
	\ln ((\coth y)^\delta(1-{\frac y x})+{\frac y x})\leq \theta(y)C(y)-\frac {(\theta(y))^2(C(y))^2} 3 \,.
	\end{equation}
	Combining \eqref{psi_2(y)-part1} and \eqref{psi_2(y)-part2}, we have
	\begin{align*}
	\ln \psi_2(y) &\leq -\psi_3'(x)x (1-\frac y x)+\frac {\psi_3''(x)x^2}2 (1-\frac y x)^2+\gamma(1+\epsilon)\theta(y)C(y)-\gamma(1+\epsilon)(\theta(y))^2\frac {(C(y))^2} 3 \\
	&=\theta(y)\gamma(1+\epsilon)[-\frac {\theta(y)}3  (C(y))^2+C(y)-\frac 1 {\gamma(1+\epsilon)}(\psi_3'(x)x-\frac {\psi_3''(x)x^2}2 \theta(y))]\,.
	\end{align*}
	We wish to show that the factor in the square bracket above is less than some constant $-\eta$, where $\eta>0$ only depends on $\beta$, i.e. for any $y\in [\frac x 2, x]$
	$$-\frac {\theta(y)}3  (C(y))^2+C(y)-\frac 1 {\gamma(1+\epsilon)}(\psi_3'(x)x-\frac {\psi_3''(x)x^2}2 \theta(y))\leq-\eta\,.$$
	Set $c:=\frac {\psi_3''(x_0)x_0^2}{2\gamma}$. Since $|x-x_0|\leq\epsilon^2$, and $\epsilon$ can be made arbitrarily small, we only need to show that for some constant $\eta_1>0$ which only depends on $\beta$, for any $y\in [\frac x 2, x]$
	$$-\frac {\theta(y)} 3 (C(y))^2+C(y)-1+c\theta(y)\leq-\eta_1\,.$$
	
	To do this, we let $q(s):=-\frac {\theta(y)} 3 s^2+s-1+c\theta(y)$. Solving the quadratic equation $q(s)=0$ with respect to $s$, we get the smaller root (since $\theta(y)\in[0,1/2]$ for $y\in [\frac x 2, x]$ and $c>-1/3$, $q(s)$ always has two roots)
	$$r(y):=\frac {-1+\sqrt{1-\frac {4\theta(y)}3 (1-c\theta(y))}}{-\frac {2\theta(y)} 3}\,, \text{ for } y\in [\frac x 2, x]\,.$$
	We claim that we only need to show that for any $\frac x 2\leq y\leq x$, $C(y)\leq r(y)-\eta_2$ for some constant $\eta_2>0$ which only depends on $\beta$. Indeed, if this holds true, then from $q'(s)=-\frac {2\theta(y)} 3 s+1$, we see that $q'(C(y))=-\frac 2 3 \theta(y)C(y)+1\geq 0.5$ and $q'(C(y)+\eta_2)=q'(C(y))-\frac {2\theta(y)} 3 \eta_2\geq 0.5-\frac 2 3 \eta_2$. Consequently $0=q(r(y))\geq q(C(y)+\eta_2)\geq q(C(y))+\eta_2(0.5-\frac 2 3 \eta_2)$ and we can take $\eta_1=\eta_2(0.5-\frac 2 3 \eta_2)$.
	
	To this end, we first point out that $r(y)$ is convex in $y$ if $c<\frac 1 3$, $r(y)$ is concave in $y$ if $c>\frac 1 3$ and $r(y)\equiv 1$ if $c=\frac 1 3$. This can be seen by observing that $r=\frac {-1+\sqrt{1-\frac {4\theta}3 (1-c\theta)}}{-\frac {2\theta} 3}$ is the inverse function of $\theta=\frac {r-1}{\frac {r^2}3 -c}$, whose properties such as monotonicity and convexity are not hard to justify. Now if $c<\frac 1 3$, then by convexity of $r(y)$, we have
	$$r(y)\geq r'(\frac {3x} 4)(y-\frac {3x} 4)+r(\frac {3x} 4):=t(y)$$
	where $t(y)$ can be computed as
	$$t(y)= -\frac 1 x(\frac {120}{\sqrt{3c+24}}-24)(y-\frac {3x} 4)+6-\sqrt{3c+24}\,.$$
	Since $C(y)$ is convex in $y$, we only need to have $t(x) \geq C(x)+\eta_2$ and $t(\frac x 2) \geq C(\frac x 2)+\eta_2$, i.e.,
	\begin{equation}
	-\frac {30}{\sqrt{3c+24}}+12-\sqrt{3c+24} \geq (\coth x)^\delta-1+\eta_2
	\label{req-1}
	\end{equation}
	and
	\begin{equation}
	\frac {30}{\sqrt{3c+24}}-\sqrt{3c+24} \geq (\coth \frac x 2)^\delta-1+\eta_2\,.
	\label{req-2}
	\end{equation}
	If $c=\frac 1 3$, then $r(y)\equiv 1$, which is a degenerate case. If $c>\frac 1 3$, then since $r(y)$ is concave in $y$, we only need to have $r(x) \geq C(x)+\eta_2$ and $r(\frac x 2) \geq C(\frac x 2)+\eta_2$, i.e.,
	\begin{equation}
	1\geq (\coth x)^\delta-1+\eta_2
	\label{req-3}
	\end{equation}
	and
	\begin{equation}
	3-\sqrt{3(c+1)}\geq(\coth \frac x 2)^\delta-1+\eta_2\,.
	\label{req-4}
	\end{equation}
	All of the inequalities \eqref{req-1}, \eqref{req-2}, \eqref{req-3} and \eqref{req-4} boil down to comparisons of constants which only involve $x_0$ (since $|x-x_0|\leq\epsilon^2$ and $\epsilon$ can be made arbitrarily small), so we have finally shown that \eqref{eq-4} is bounded by $C_{13}N(\sinh x)^{\beta N}(\cosh x)^{(1-\beta)N}$.
	
	Combining \eqref{eq-1}, \eqref{eq-2}, \eqref{eq-5}, \eqref{eq-4} and the fact that $F(v_0,v_L)\leq N(f(x))^{N-1}f'(x)$ when $d=0$, we conclude that $\sum_{d=0}^{L\over2} \sum_{d_1+d_2=d}F(v_{d_1},v_{L-d_2})\leq C_2'N(\sinh x)^{\beta N}(\cosh x)^{(1-\beta)N}$ for some $C_2'>0$.\qed
\end{proof}

\begin{lemma}\label{lem-less-than-half-segment-general}
	For any sufficiently small but fixed number $\epsilon>0$, there exist $C_3'>0$ and an integer $N'>0$ which both depend only on $\epsilon$, such that for all $|x - x_0| \leq \epsilon^2$, $N>N'$, any good path $P$ and any $j$ we have $\sum_{i=1}^{\frac L 2}F(v_j,v_{j+i})\leq 1+{\frac {C_3'} N}$.
\end{lemma}
\begin{proof}
	The proof can be carried out in the same manner as that of Lemma~\ref{prop-less-than-half-segment}, except that the role of $\alpha+\epsilon_3$ in Case (e) there is now played by $\frac {\gamma+\epsilon_3} {2g(1/2)}$. We thus omit the details.\qed
\end{proof}

\begin{corollary}
	\label{cor-second-moment-general}
	For any sufficiently small but fixed number $\epsilon>0$, there exist $C_4'>0$ and an integer $N'>0$ which both depend only on $\epsilon$,
	such that for all $|x - x_0| \leq \epsilon^2$ and $N>N'$
	$$\E Z_{N,x,*}^2\leq (C_4'N(\sinh x)^{\beta N}(\cosh x)^{(1-\beta)N}+C_4')N(\sinh x)^{\beta N}(\cosh x)^{(1-\beta)N}\,.$$
\end{corollary}

\begin{proof}
	This follows from Lemmas~\ref{lem-longest-segment-general} and \ref{lem-less-than-half-segment-general} in the same manner as Corollary \ref{cor-second-moment} follows from Lemmas~\ref{prop-longest-segment} and \ref{prop-less-than-half-segment}.\qed
\end{proof}

\begin{proposition}\label{prop-augmenting-general}
	There exists $0\leq K'<1$ such that, if $\liminf\limits_{N\to\infty}\P(Z_{N,x_c+\epsilon_N}>0)\geq C$ for some constant $C\geq 0$ whenever $N\epsilon_N\to \infty$, then whenever $N\epsilon_N\to \infty$ we have
	$$\liminf\limits_{N\to\infty}\P(Z_{N,x_c+\epsilon_N}>0)\geq 1-(1-C)K'\,.$$
\end{proposition}

\begin{proof}
	The basic idea is the same as Proposition \ref{eq-augmenting}. Fix a large integer $M$. We first choose vertices $A_1, \ldots A_M, B_1, \ldots, B_{M}$ and $C_1, \ldots, C_M, D_1, \ldots, D_M$ such that for $1\leq i \leq M$:
	\begin{itemize}
		\item The only coordinate at which $A_{i-1}$ and $A_i$ differ is $a_i$. The only coordinate at which $B_{i-1}$ and $B_i$ differ is $b_i$. The only coordinate at which $C_{i-1}$ and $C_i$ differ is $c_i$. The only coordinate at which $D_{i-1}$ and $D_i$ differ is $d_i$ (set $A_0=C_0=\vec{0}_N$ and $B_0=D_0=(\vec 1_{\beta N}, \vec 0_{N-\beta N})$ for convenience). 
		\item All of the $4M$ coordinates $a_i$, $b_i$, $c_i$ and $d_i$ are different and are among the first $\beta N$ coordinates.
		\item $X(A_i), X(C_i)\in [\frac {(i-1)\epsilon_N} {4M}, \frac {i\epsilon_N} {4M}]$ and  $X(B_i), X(D_i)\in [x-\frac {i\epsilon_N} {4M}, x-\frac {(i-1)\epsilon_N} {4M}]$.
	\end{itemize}
	Since $N\epsilon_N \to \infty$, this can be achieved with probability $1-o_N(1)$.
	
	Now let $M_2={\frac {(1-\beta)}\beta}2M$, and select distinct coordinates $e_1,e_2,\cdots, e_{M_2}$ and $f_1, f_2,\cdots, f_{M_2}$ arbitrarily among the last $(1-\beta)N$ coordinates. Let $\tilde H_1$ be the $(N-2M-M_2)$ dimensional sub-hypercube formed by $A_M$ and $B_M$ with the coordinates $e_1,e_2,\cdots, e_{M_2}$ being 0, i.e., 
	$$\tilde H_1 = \{\sigma\in H_N: \sigma_{e_i} = 0 \mbox{ for all } 1\leq i\leq M_2, \sigma_{a_i} = 1 \mbox{ for all } 1\leq i\leq M, \sigma_{b_i} = 0\mbox{ for all } 1\leq i\leq M\}\,.$$
	Similarly, let $\tilde H_2$ be the $(N-2M-M_2)$ dimensional sub-hypercube formed by $C_M$ and $D_M$ with the coordinates $f_1, f_2,\cdots, f_{M_2}$ being 0, i.e., 
	$$\tilde H_2 = \{\sigma\in H_N: \sigma_{f_i} = 0 \mbox{ for all } 1\leq i\leq M_2, \sigma_{c_i} = 1 \mbox{ for all } 1\leq i\leq M, \sigma_{d_i} = 0\mbox{ for all } 1\leq i\leq M\}\,.$$
	Let $H_2' = \tilde H_2\setminus \tilde H_1$. Denote by $p_{\tilde H_1}$ and $p_{H_2'}$ the probabilities that there is an accessible path in $\tilde H_1$ (from $A_M$ to $B_M$) and $H_2'$ (from $C_M$ to $D_M$) respectively. Since $\tilde H_1$ and $H_2'$ are disjoint, by independence we have $\P(Z_{N,x_c+\epsilon_N}>0)\geq1-(1-p_{H_1})(1-p_{H_2'})-o_N(1)$. From the construction above it is clear that we are reduced to accessibility percolation of dimension $(N-2M-M_2)$ (with the same $\beta$) with $x\geq x_c+\epsilon_N/2$, in either $\tilde H_1$ (from $A_M$ to $B_M$) or $\tilde H_2$ (from $C_M$ to $D_M$). Thus, 
	$$p_{\tilde H_1}\geq \P(Z_{N-2M-M_2,x_c+\epsilon_N/2}>0)\geq C-o_N(1)\,.$$
	
	To show that $p_{H_2'}$ is bounded from below by a positive constant $1-K'$, we only consider the good path in $\tilde H_2$ (from $C_M$ to $D_M$) which updates each of coordinates $a_1$ and $b_1$ precisely once and $b_1$ is updated before $a_1$. Such paths must be contained in $H_2'$. Clearly, the number of such accessible paths has second moment less than  
	$\E Z^2_{N-2M-M_2,x_c+\epsilon_N/2,*}$ and first moment within an absolute multiplicative constant of $\E Z_{N-2M-M_2,x_c+\epsilon_N/2,*}$ (or $M_{N-2M-M_2,\beta,x_c+\epsilon_N/2}$). Combined with Lemma~\ref{lem-first-moment-general} and Corollary~\ref{cor-second-moment-general}, this yields that $p_{H_2'} \geq 1-K'-o_N(1)$ for some constant $K'<1$. This completes the proof of the proposition.\qed
\end{proof}

\paragraph{Proof of \eqref{eq-mainthm-lower}: general case} Applying Proposition~\ref{prop-augmenting-general} recursively (starting from $C=0$) completes the proof of \eqref{eq-mainthm-lower}.\qed

\paragraph{Proof of \eqref{eq-mainthm-critical}: general case} The proof is basically the same as in the antipodal case except that for the upper bound, the role of $\sinh(x)$ is now played by $f(x)=(\sinh x)^{\beta}(\cosh x)^{1-\beta}$.\qed

\begin{acknowledgements}
The author would like to thank his advisor Jian Ding for suggesting the problems, numerous helpful discussions as well as careful editing and suggestions for writing. The author would like to thank Subhajit Goswami for pointing out \eqref{eq-second-moment} in a useful discussion and for his useful comments on an early version of the manuscript. The author would also like to thank Professor Steven Lalley for his advice on writing which greatly improved the manuscript.
\end{acknowledgements}

%\bibliography{references}

\end{document}